\documentclass[12pt,a4paper]{article}

\usepackage[utf8]{inputenc}
\usepackage{amsmath}
\usepackage{amsthm} 

\usepackage{parskip} 

\usepackage{amsfonts}
\usepackage{amssymb}
\usepackage{graphicx}
\usepackage{multirow}
\usepackage{geometry}
\usepackage{subcaption}

\newtheoremstyle{exampstyle}
{3mm} 
{} 
{\itshape} 
{} 
{\bfseries} 
{.} 
{.5em} 
{} 

\theoremstyle{exampstyle}

\newtheorem{theo}{Theorem}
\numberwithin{theo}{section} 
\newtheorem{defi}[theo]{Definition}
\newtheorem{lem}[theo]{Lemma}

\newtheorem{cor}[theo]{Corollary}
\newtheorem{con}[theo]{Conjecture}
\newtheorem{rem}[theo]{Remark}

\usepackage[font=small,format=plain,labelfont={bf,up},labelsep=period,textfont={it},margin=0pt]{caption}

\usepackage{cite}

\usepackage{url}
\usepackage{color}
\usepackage[title]{appendix}

\begin{document}

\begin{center}
\section*{Fungal tip growth arising through a codimension-1 global bifurcation}
\large{T.G. de Jong, A.E. Sterk, H.W. Broer}
\end{center}

\begin{small} \textbf{Abstract.}
Tip growth is a growth stage which occurs in fungal cells. During tip growth, the cell exhibits continuous extreme lengthwise growth while its shape remains qualitatively the same. A model for single celled fungal tip growth is given by the Ballistic Ageing Thin viscous Sheet (BATS) model, which consists of a 5-dimensional system of first order differential equations. The solutions of the BATS model that correspond to fungal tip growth arise through a codimension-1 global bifurcation in a 2-parameter family of solutions. In this paper we derive a toy model from the BATS model. The toy model is given by 2-dimensional system of first order differential equations which depend on a single parameter. The main achievement of this paper is a proof that the toy model exhibits an analogue of the codimension-1 global bifurcation in the BATS model. An important ingredient of the proof is a topological method which enables the identification of the bifurcation points. Finally, we discuss how the proof may be generalized to the BATS model.
\end{small}


\section{Introduction}

Tip growth is a growth stage of a biological cell. During tip growth the cell exhibits extreme lengthwise growth while its shape remains qualitatively the same and the cell's tip velocity remains approximately constant, see Figure \ref{fig:tipgrowth}. 

\begin{figure}[h]
\begin{center}
\includegraphics[width=11cm]{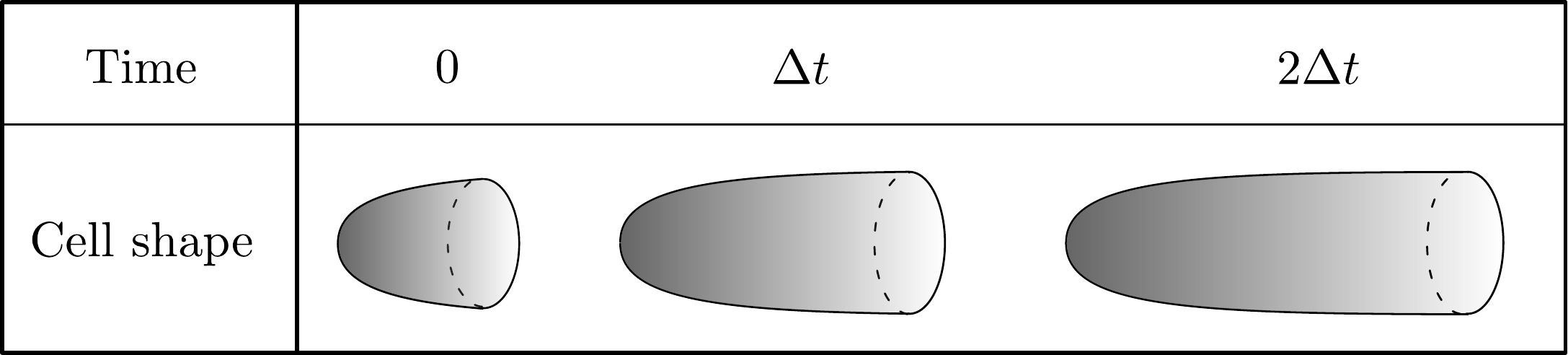}
\caption{Qualitative illustration of tip growth depicting the cell shape at time steps $\Delta t$. During tip growth the cell's tip moves at a constant speed and the shape remains qualitatively the same. \label{fig:tipgrowth}}
\end{center}
\end{figure}

\begin{figure}[t]
\begin{center}
\includegraphics[width=9.5cm]{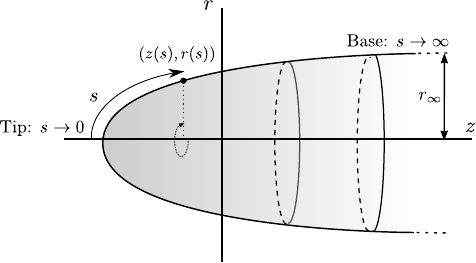}
\caption{In the BATS model the cell wall is a surface of revolution where the generating curve in the $(r,z)$-plane is parametrised by its arc length $s$ to the tip. Axial symmetry is a standard assumption in fungal tip growth models \cite{BG89,EGG11,GO10,TIN06}. The tip of the cell is given by the limit $s \rightarrow 0$.  The cell is assumed to be infinitely long so that the base of the cell corresponds to the limit $s \rightarrow \infty$, and the limiting base radius is denoted by $r_\infty$. \label{fig:cellshape}}
\end{center}
\end{figure}

\paragraph{Modelling considerations.}

In \cite{JONBIO} the Ballistic Ageing Thin viscous Sheet (BATS) model for single celled fungal tip growth is presented. The BATS model incorporates material properties of the cell wall to model the cell's growth. The BATS model is studied in a co-moving frame which removes the time dependency. The governing equations of the BATS model are given by a 5-dimensional system of first order differential equations. The independent variable of the BATS model is the arc length to the tip denoted by $s$, see Figure \ref{fig:cellshape}. The main problem is to find solutions which correspond to the cell shape in Figure \ref{fig:cellshape}. These solutions are called \emph{steady tip growth solutions} since they describe continuous fungal growth close to the cell's tip.

In \cite{JONIEEE} the first two authors introduced a family of solutions which allow for a parametrization by two parameters such that the numerically computed steady tip growth solutions arise through a codimension-1 global bifurcation. This bifurcation is non-standard since it does not involve a stability change upon variation of parameters. The existence of this global bifurcation is hard to confirm analytically. Therefore, we will derive a toy model from the BATS model which captures the phenomenology in an understandable way, see Figure \ref{fig:3blocks} for a schematic. \newpage

\begin{figure}[t]
\begin{center}
\includegraphics[width=13cm]{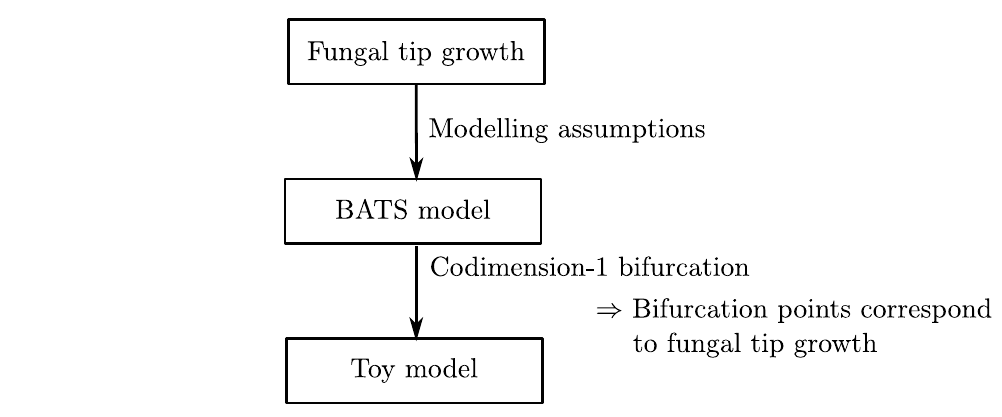}
\caption{Modelling schematic: From the biology of fungal tip growth modelling assumptions were derived which led to the Ballistic Ageing Thin viscous Sheet (BATS) model in \cite{JONBIO}. A global codimension-1 bifurcation in the governing equations of the BATS model was numerically indentified in \cite{JONIEEE}. The bifurcation points correspond to solutions describing fungal tip growth. \label{fig:3blocks}}
\end{center}
\end{figure}

It follows from analysis that the cell shape corresponding to steady tip growth solutions are characterised by the radius variable $r$ and its first derivative $\rho = r'$ with respect to the arc length $s$. Therefore, the toy model is designed to depend only on these two variables and a single parameter. Figure \ref{fig:rhorsteady} qualitatively describes the variables $r$ and $\rho$ corresponding to steady tip growth solutions. Figure \ref{fig:bifu_sol} illustrates how steady tip growth solutions arise through a codimension-1 bifurcation in both the BATS model and the toy model. The bifurcation in the BATS model forms the basis of the toy model presented in this paper.

\begin{figure}[h]
\begin{center}
\includegraphics[width=7cm]{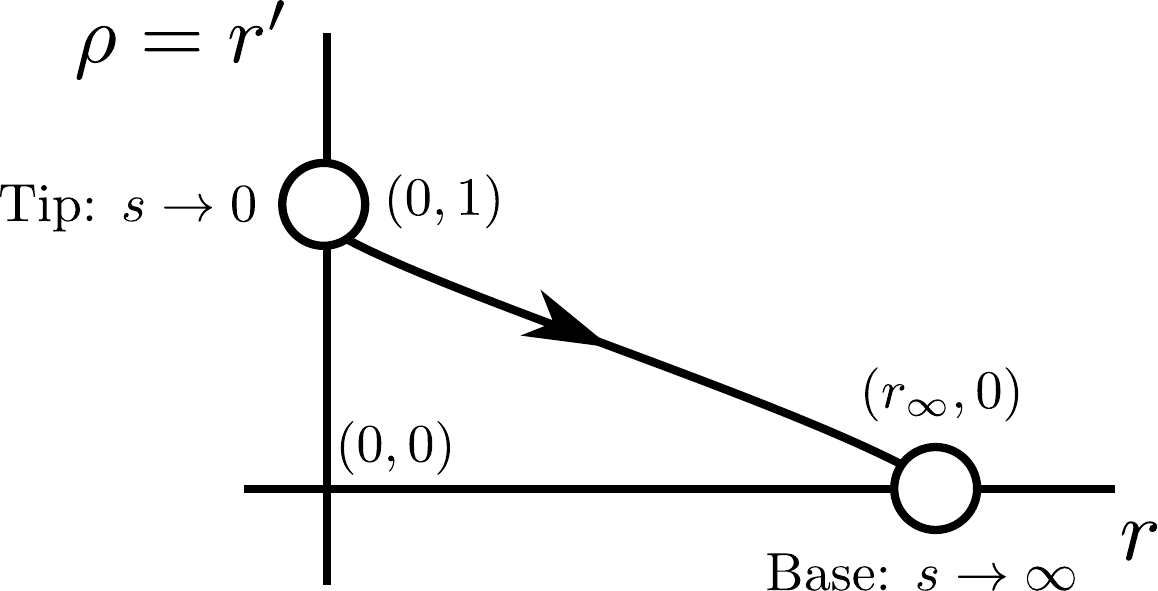}
\caption{Steady tip growth solutions in the $(\rho,r)$-plane: The independent variable is the arc length $s$ to the tip. The dependent variables $\rho$ and $r$ are related by $\rho=r'$ where the prime denotes taking the derivative with respect to $s$. Analysis reveals that the white circles correspond to the tip and base as indicated in Figure \ref{fig:cellshape}. At these limit points the governing ODE of the BATS model is not defined.   \label{fig:rhorsteady}}
\end{center}
\end{figure}

\begin{figure}[t]
\begin{center}
\includegraphics[width=9cm]{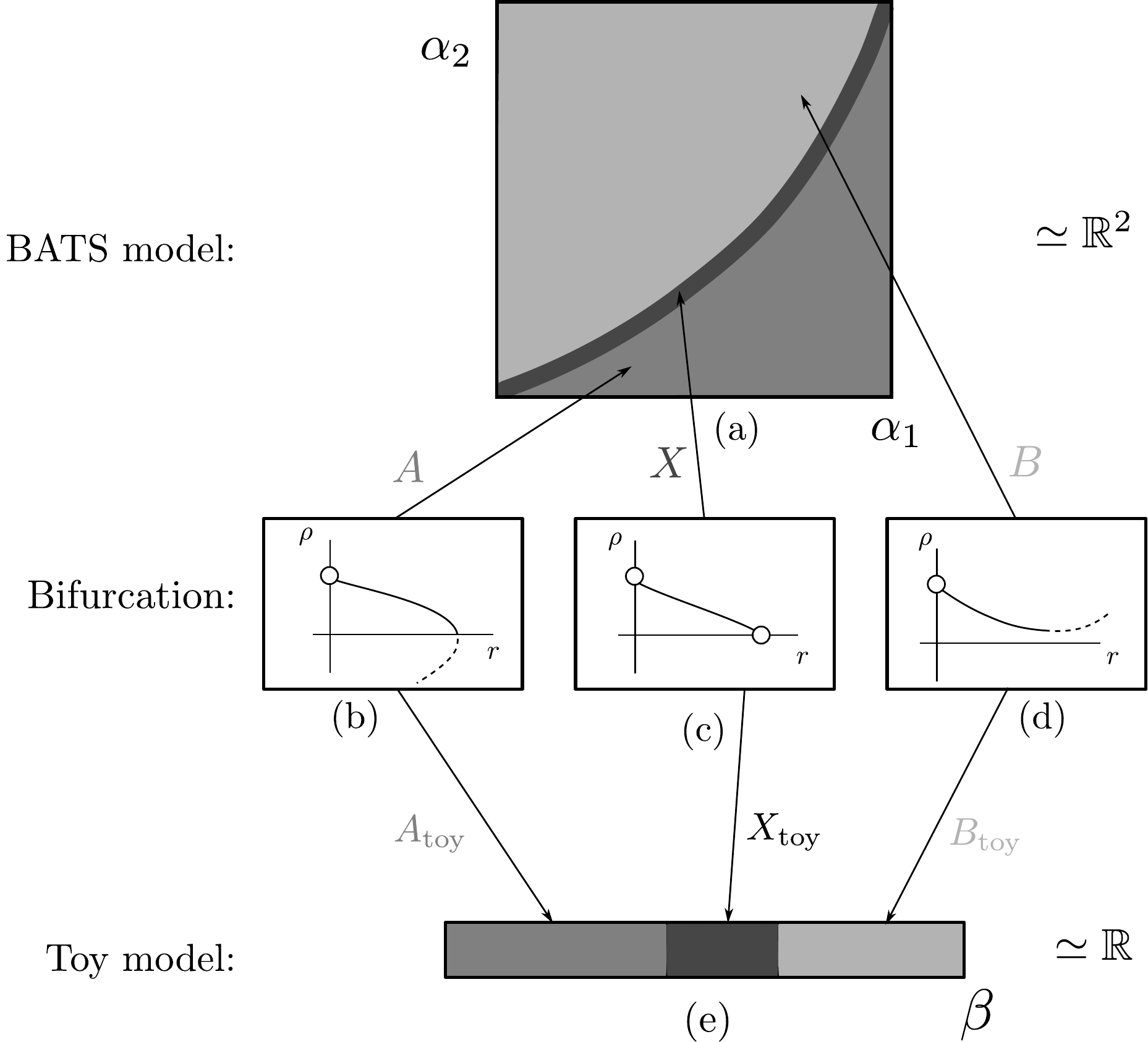}
\caption{Codimension-1 bifurcation with corresponding bifurcation diagram:  (a) and (e) are  qualitative representations of the parameter space corresponding to the parametrised solution set of the BATS model and toy model, respectively. In (b)-(d) the codimension-1 bifurcation is displayed. The white circles correspond to limit points. The solutions in  (b) are described by a sign change in $\rho$ and correspond to the parameter set $A$ in (a) and the parameter set $A_{\rm toy}$ in (e).  The solutions in (d) are described by a sign change in $d\rho/ds$ and correspond to the parameter set $B$ in (a) and the parameter set $B_{\rm toy}$ in (e). The solutions in (c) have no sign change in $\rho,d\rho/ds$ and correspond to the parameter set $X$ in (a) and $X_{\rm toy}$ in (e). \label{fig:bifu_sol}}
\end{center}
\end{figure}

\paragraph{Mathematical analysis.}

The existence of the codimension-1 bifurcation in the toy model will be proved by means of a topological method. The parameter set corresponding to the toy model is given by $\mathbb{R}_+$. A solution set which depends continuously on the parameter set is identified.  This solution set is used to define two open, non-empty, disjoint sets: $A_{\rm toy} \subset \mathbb{R}_{>0}$ corresponding to Figure \ref{fig:bifu_sol}b and $B_{\rm toy} \subset \mathbb{R}_+$ corresponding to Figure \ref{fig:bifu_sol}d. It is shown that $A_{\rm toy}$, $B_{\rm toy}$ are open, non-empty, disjoint sets and that $X_{\rm toy} := \mathbb{R}_{>0} \backslash (A_{\rm toy} \cup B_{\rm toy})$ corresponds to Figure \ref{fig:bifu_sol}c. A local analysis cannot be used to determine  the bifurcation points $X_{\rm toy}$ since the bifurcation is of a global nature. In the literature this topological method is referred to as topological shooting since the topological sets determines how to `shoot' trajectories to find the desired solution \cite{HAS12}. For another application of topological shooting see \cite{PEL95}.

\paragraph{Overview.}

This paper is structured as follows. In Section \ref{sec:matdes} we give a mathematical description of the tip growth cell shape.  In Section \ref{sec:toymaintheo} we present the toy model with main theorems and proofs. In Section \ref{secgov} we review the biological model for fungal cell growth, called the Ballistic Ageing Thin viscous Sheet (BATS) model. In Section \ref{sec:toyandbio} the relation between the toy model and the BATS model is explained. Finally, the extension of the proof for the toy model to the BATS model is discussed in Section \ref{sec:disc}. The technical proofs can be found in the Appendices.


\section{Mathematical description of tip growth cells \label{sec:matdes}}

The governing equations of the BATS model are given by a system of 5-dimensional non-linear first order ODEs. The BATS model should model fungal tip growth. Therefore, the BATS model is validated by proving the existence of solutions which resemble fungal tip growth, the so-called steady tip growth solutions. The aim of the toy model is to prove the existence of a toy analogue of steady tip growth solutions. As a prelude to the toy model we will give a mathematical description and heuristic explanation of steady tip growth solutions. Besides variables describing the cell's shape the BATS model has variables for cell wall ageing and cell wall thickness which in total amounts to 5 variables.
 An overview of the BATS model is given in Section \ref{secgov}. Since the toy model's dependent variables are only related to the cell's shape we describe steady tip growth solutions in terms of the cell shape variables.   

The cell wall is a surface of revolution where the generating curve in the $(r,z)$-plane is parametrised by its arc length $s$ to the tip, see Figure \ref{fig:cellshape}. For the ODE of the toy model the independent variable is $s$  and the dependent variables are $r$ and the first derivative of $r$ denoted by $\rho$. Observe that from the definition of the arc length $s$ it follows that $r'^2 + z'^2 = 1$, where the prime denotes the derivative with respect to the arc length $s$. By Figure \ref{fig:cellshape} we observe that $z' > 0$. Consequently, we get the equality
\begin{align}
z'= \sqrt{ 1-  r'^2}. \label{dzdsarc}
\end{align} 
Using \eqref{dzdsarc} it follows that the variables $\rho,r$ give a full description of the cell shape upto an initial condition in the $z$-variable.
 
The cell shape in Figure \ref{fig:cellshape} is described by two local conditions at the tip, $s \rightarrow 0$, global conditions, $0<s<\infty$ and a local condition at the base, $s \rightarrow \infty$:

\begin{itemize}
\item[S1] \textbf{Tip limits:}  The following limits are satisfied:
\begin{align*}
\lim_{s \rightarrow 0} \rho(s) =1, \; \;\; \lim_{s \rightarrow 0} r(s) =0,  \;\;\; \lim_{s \rightarrow 0 }\frac{\sqrt{1-\rho(s)^2}}{r(s)} = \eta_0 >0 .
\end{align*}

\textit{Heuristic explanation:} In Figure \ref{fig:cellshapetip} we displayed the limiting conditions at the tip for $(z,r)$ as we would expect from the cell shape in Figure \ref{fig:cellshape}.  Using \eqref{dzdsarc} we obtain the limiting condition for $\rho$. The last limit follows from the principal curvatures. We consider the $s$- and $\phi$-direction with $\phi$ the angular co-ordinate, see Figure \ref{fig:cellshapetip}. The principal curvature are given by  
\begin{align*}
\kappa_s = -\frac{r''}{z'}, \qquad \kappa_\phi =  \frac{z'}{r}, 
\end{align*}
see \cite{HW94} for the derivation.  The tip is locally concave, therefore, we must require that the principal curvatures are positive at the tip. Since the tip of the cell intersects with the axis of revolution it follows that the tip is an umbilical point:
\begin{align*}
\lim_{s \rightarrow 0} \kappa_s(s) = \lim_{s \rightarrow 0} \kappa_\phi(s). 
\end{align*} 
Hence, we only need to require that $\kappa_\phi$ is positive at the tip. Using \eqref{dzdsarc} we re-write $\kappa_\phi$ in terms of $\rho,r$ which yields the last limit in S1.
\end{itemize}

\begin{figure}[h]
\begin{center}
\includegraphics[width=7cm]{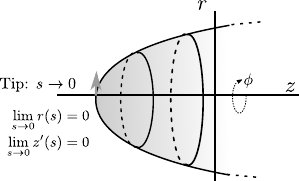}
\caption{Tip shape: The tip corresponds to $s \rightarrow 0$. The grey vector denotes the tangent vector at the tip. The limits in the figure follow from the cell shape. The cell is axially symmetric. The angular co-ordinate is given by $\phi$. \label{fig:cellshapetip}}
\end{center} \end{figure} 

\begin{itemize}
\item[S2] \textbf{Analyticity in $r^2$:} There exists $s_0>0$ and $G \in C^{\omega}( (-a,a), \mathbb{R}_{>0})$ with $a=r_*(s_0)^2$  such that  
\begin{align*}
\rho_*(s) = G(r_*(s)^2) \; \; \; \forall s \in (0,s_0).
\end{align*}

\textit{Heuristic explanation:} We expect $z$ to be an even function close to the tip when parametrized in $r$ due to axial symmetry and smoothness close to the tip. Then, using \eqref{dzdsarc} the condition S2 follows.  
\end{itemize}

\begin{itemize}
\item[S3] \textbf{Global constraints:} For all $s \in \mathbb{R}_{>0}$ the following inequalities are satisfied: 
\begin{align*}
\rho'(s) < 0, \; \; \; \rho(s) >0. 
\end{align*}

\textit{Heuristic explanation:} The generating curve in the $(r,z)$-plane is concave and monotone. Therefore, we require that $r''(s)>0, r'(s)>0$ for all $s \in \mathbb{R}_{>0}$ which is equivalent to S3
\end{itemize}
\begin{itemize}
\item[S4] \textbf{Base limits:} The following limits are satisfied:
\begin{align*}
\lim_{s \rightarrow \infty} \rho(s) =0, \; \;\; \lim_{s \rightarrow \infty} r(s) =r_{\infty}>0.
\end{align*}
\textit{Heuristic explanation:} We require S4 since we expect that the cell converges to a fixed width at the base. 
\end{itemize}

If $\rho,r$ do not satisfy all conditions S1-S4 then they do not describe idealized tip growth.


\section{Toy model and main theorems \label{sec:toymaintheo}}

In this section the toy model and main theorems with proofs are presented. In Section \ref{secgov} the BATS model is revised and in Section \ref{sec:toyandbio} the toy model is derived from the BATS model.

The BATS model has a functional dependency on a one dimensional smooth function $\mu$ called the viscosity function. The viscosity function is not specified since it is expected to be fungus dependent. In the toy model $g \in C^{\omega}(\mathbb{R})$ satisfying $g(v)>0$ for all $v \in \mathbb{R}_{>0}$ will take the role of the viscosity function. The governing equations of the toy model are given by
\begin{gather}
\begin{aligned}
\rho' &= \frac{3}{2} \frac{1-\rho^2}{{r}} \left( -  1+  \frac{ \sqrt{1-\rho^2}  (   \beta  {r}^2 g( {r}^2) + \rho    )   }{{r}}    \right) , \\
r' &=  \rho, 
\end{aligned}  \label{toygov}
\end{gather}
where $g \in C^{\omega}(\mathbb{R})$ satisfies $g(v)>0$ for all $v \in \mathbb{R}_{>0}$ and $\beta \in \mathbb{R}_{>0}$. The phase space is given by
\begin{align*}
M_0 = \{ (\rho,r) \in (-1,1) \times \mathbb{R}_{>0} \}. 
\end{align*} 
We refer to the dynamical system corresponding to \eqref{toygov} as the \emph{toy model}. The solutions of the toy model \eqref{toygov} which correspond to tip growth follow from Section \ref{sec:matdes}:

\begin{defi}[Toy steady tip growth solution]  $(\rho,r)$ is a \emph{toy steady tip growth solution} if it is a solution of the toy model \eqref{toygov} that satisfies conditions S1-S4.
\label{def:toy}
\end{defi}

We now present the main theorems for the toy model \eqref{toygov}:

\begin{theo}[Existence of toy steady tip growth solutions] Let $g \in C^{\omega}(\mathbb{R})$  be a function such that 
\begin{gather}
\begin{aligned}
g(v)>0, & \qquad    \frac{\partial g(v)}{ \partial v}\geq 0, \qquad \frac{\partial^2 (v^2 g(v^2))}{ \partial v^2}>0 \qquad \forall v \in \mathbb{R}_{>0} , \\
\lim_{v \rightarrow \infty }v g(v^2) &= \infty .
\end{aligned}  \label{gfuncc}
\end{gather} Then there exists a $\beta>0$ such that the toy model \eqref{toygov} has a unique  toy steady tip growth solution as specified in Definition \ref{def:toy}.  \label{theo:maintoy}
\end{theo}

\begin{theo}[Topology of parameter set] Let $g \in C^{\omega}(\mathbb{R})$ satisfy \eqref{gfuncc}.   The set of all $\beta >0$ such that the toy model \eqref{toygov} has a toy steady tip growth solution as specified in Definition \ref{def:toy} is a closed set with empty interior.  \label{theo:toptoy}
\end{theo}

The conditions in \eqref{gfuncc} are clearly satisfied when $g$ is a positive constant function. Hence, Theorem \ref{theo:maintoy} does not concern an empty set of functions. Throughout this  section we consider the toy model \eqref{toygov} with $g \in C^{\omega}(\mathbb{R})$ satisfying \eqref{gfuncc}.

The technique used to prove Theorem \ref{theo:toptoy} will rely on the planarity of the toy model \eqref{toygov}. Hence, extending the proof of Theorem \ref{theo:toptoy} to the five dimensional BATS model would require additional properties.

Theorems \ref{theo:maintoy} and \ref{theo:toptoy} rely on the existence of a family of solutions, called toy tip solutions, which are solutions which satisfy the properties of a toy steady tip growth solutions as given by Definition \ref{def:toy} on an $s$-interval $(0,s_0)$. These solutions undergo the bifurcation in Theorem \ref{theo:toptoy}. We present the definition of toy tip solutions with a corresponding existence and uniqueness result in Section \ref{sec:toytip}. The corresponding proof is technical and will be presented in Appendix \ref{sec:exuntoytip}. Proof overviews of Theorems \ref{theo:maintoy} and \ref{theo:toptoy} are presented in Section \ref{sec:maintoy} and Section \ref{sec:toptoy}, respectively.  Proofs of the lemmas for Theorems \ref{theo:maintoy} and \ref{theo:toptoy} are  presented in Appendix \ref{sec:tech}. 


\subsection{Toy tip solutions}
\label{sec:toytip}

We define the solution set which will undergo the global bifurcation:

\begin{defi}[Toy tip solution] A solution $(\rho,r)$ of the toy model \eqref{toygov} is a \emph{toy tip solution} if and only if it satisfies S1,S2 and if there exists an $s_0 >0$ such that
\begin{align}
\rho'(s)<0, \qquad \rho(s) >0 \qquad \forall s \in (0,s_0).  \label{localtoytipineq}
\end{align}
\label{def:toytip}
\end{defi}
It follows from Definition \ref{def:toy} and Definition \ref{def:toytip} that
\begin{align}
\text{toy steady tip growth solutions } \subset \text{ toy tip solutions}. \label{subset2}
\end{align}
The property \eqref{subset2} is crucial since we will prove the existence of toy steady tip growth solutions as given by Definition \ref{def:toy} as the result of a bifurcation of toy tip solutions as given by Definition \ref{def:toy}.  

The proofs of Theorems \ref{theo:maintoy} and \ref{theo:toptoy} rely on the construction of toy tip solutions.  The construction of toy tip solutions relies on a change of variables such that in the new variables proving the existence  of toy tip solutions for a $\beta$ is equivalent to proving the existence of an unstable manifold. Observe that the toy model \eqref{toygov} does not have an equilibrium corresponding to the tip limits condition S1.   The uniqueness of toy tip solutions will be equivalent to showing that the unstable manifold is 1-dimensional. Due the technicalities involved the toy tip solution construction theorem is presented in Appendix \ref{sec:exuntoytip}. 
Denote by $(\rho_\beta,r_\beta)$ the tip solution corresponding to the parameter $\beta$.

 \begin{cor}
For all $\beta \geq 0 $ there exists a unique toy solution $x_\beta$ specified by Definition \ref{def:toytip}. In addition, the map $F_{\rm toy}: \beta \mapsto (\rho_\beta,r_\beta)$ is continuous in $\beta$. \label{cor:toytip}
\end{cor}

The proof of Corollary \ref{cor:toytip} is given in Appendix \ref{p:cor:toytip}. Observe that in Corollary \ref{cor:toytip} we consider $\beta \geq 0 $ while for the toy model \eqref{toygov} we only considered $\beta > 0 $. For $\beta =0 $ the ODE \eqref{toygov} becomes degenerate which is used to prove the non-emptiness of a set required for our topological argument.


\subsection{Overview proof of Theorem \ref{theo:maintoy}}
\label{sec:maintoy}

Denote the toy tip solution corresponding to $\beta$ by $(\rho_\beta,r_\beta)$. We consider the following subsets of the parameter space:
\begin{gather}
\begin{aligned}
A_{\rm toy} &:= \{  \beta  \in \mathbb{R}_{>0} \; :  \; \exists  s_0 \in \mathbb{R}_{>0}  \; \;  \rho_\beta(s) \rho_\beta'(s) <0 \; \forall s \in (0,s_0), \;  \rho_\beta(s_0) =0       \} ,\\
B_{\rm toy} &:= \{  \beta \in \mathbb{R}_{>0} \; :  \; \exists  s_0 \in \mathbb{R}_{>0}  \; \;  \rho_\beta(s) \rho_\beta'(s) <0 \; \forall s \in (0,s_0), \;  \rho_\beta'(s_0) =0   \}.
\end{aligned} \label{ABtoy}
\end{gather}
Observe that the solutions corresponding to $A_{\rm toy}$ and  $B_{\rm toy}$ are described by Figure \ref{fig:Amu} and Figure \ref{fig:Bmu}, respectively. As a result of the toy model \eqref{toygov} the solutions corresponding to $A_{\rm toy}$ and  $B_{\rm toy}$ are described by Figure \ref{fig:bifu_sol}b and Figure \ref{fig:bifu_sol}c, respectively. We define 
\begin{align}
X_{\rm toy} :=  \mathbb{R}_{>0} \setminus (A_{\rm toy} \cup B_{\rm toy}). \label{Xtoy}
\end{align}
By the definition of $A_{\rm toy},B_{\rm toy}$ in \eqref{ABtoy} it follows that $X_{\rm toy}$ satisfies S3 of Definition \ref{def:toy}, see Figure \ref{fig:Xmu}. Furthermore, if toy steady tip growth solutions exist as given by Definition \ref{def:toy} then they must correspond to $(\rho_\beta,r_\beta)$ with $\beta \in X_{\rm toy}$.

\begin{figure}[h]
\begin{center}
\includegraphics[width=4.5cm]{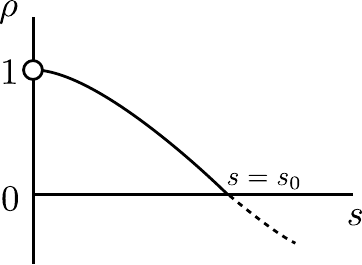}
\caption{The $\rho$-variable characterizing $A_{\rm toy}$.  The black curve corresponds to $\rho_\beta(\, \cdot \, )$ with $\beta \in A_{\rm toy}$ and $\circ$ is a limit point of the solution curve. Observe that $\rho_\beta(\, \cdot \, )$ has a zero at $s=s_0$. For $s>s_0$ the behaviour of the solution curve is not relevant for the classification. Hence, the black curve is continued with a dashed curve. \label{fig:Amu}}
\end{center}
\end{figure}

\begin{figure}[h]
\begin{center}
\includegraphics[width=4.5cm]{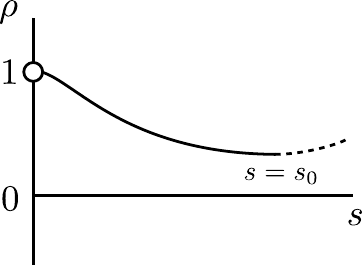}
\caption{The $\rho$-variable characterizing set $B_{\rm toy}$.  The black curve corresponds to $\rho_\beta(\, \cdot \, )$ with $\alpha \in B_{\rm toy}$ and $\circ$ is a limit point of the solution curve. Observe that $\rho_\beta'(\, \cdot \, )$ has a zero at $s=s_0$. For $s>s_0$ the behaviour of the solution curve is not relevant for the classification. Hence, the black curve is continued with a dashed curve.\label{fig:Bmu}}
\end{center}
\end{figure}

\begin{figure}[h]
\begin{center}
\includegraphics[width=4.5cm]{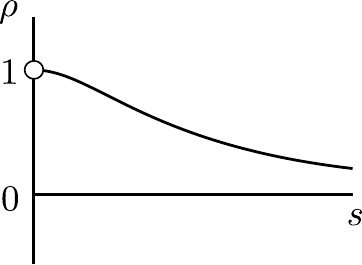}
\caption{The $\rho$-variable chacterizing $X_{\rm toy}$. The black curve corresponds to $\rho_\beta(\, \cdot \, )$ with $\beta \in X_{\rm toy}$ and $\circ$ is a limit point of the solution curve. Observe that if $\beta \in X_{\rm toy}$ then $\rho_\beta(\, \cdot \, ) \rho_\beta'(\, \cdot \, )$ has no zeroes. \label{fig:Xmu}}
\end{center} \end{figure}

\begin{lem} Let $ X_{\rm toy} \neq \emptyset$. If  $\beta \in X_{\rm toy}$ then $(\rho_\beta,r_\beta)$ is a toy steady tip growth solution as specified in Definition \ref{def:toy}. \label{lem:X}
\end{lem}

In other words, Lemma \ref{lem:X} states that if S1-S3 of Definition \ref{def:toy} are satisfied then S4 is satisfied.

\begin{lem} $A_{\rm toy}$ and $B_{\rm toy}$ are non-empty, open, disjoint sets. \label{lem:AB}
\end{lem}

\textit{Proof of Theorem \ref{theo:maintoy}}. Lemma \ref{lem:AB} implies that $X_{\rm toy} \neq \emptyset$. Then, using Lemma \ref{lem:X} and Corollary \ref{cor:toytip} the theorem follows. \hfill $\square$ 


\subsection{Overview proof of Theorem \ref{theo:toptoy}}
\label{sec:toptoy}

We consider the toy model \eqref{toygov} with $g \in C^{\omega}(\mathbb{R})$ satisfying \eqref{gfuncc}. Define 
\begin{align}
M_1 := \{ (\rho ,r ) \in  (0,1) \times \mathbb{R}_{>0}  \}. \label{M1}
\end{align}
Denote the tip solution $(\rho_\beta,r_\beta)$ restricted to the phase space $M_1$ by $(\rho_\beta,r_\beta)|_{M_1}$. Observe that $\rho_\beta$ can be parametrised in the $r$-variable since $r_\beta$ is monotone. More formally, we define:
\begin{align}
\varrho(\cdot , \beta) := \rho_\beta|_{M_1} \circ \left( r_\beta|_{M_1} \right)^{-1}. \label{varrho} 
\end{align}

\begin{lem}$\varrho$ is a smooth function satisfying
\begin{align*}
\frac{\partial  \varrho}{\partial \beta}  > 0 .
\end{align*}
\label{lem:varrho}
\end{lem}

For toy steady tip growth solutions we have another ordering resulting from S4. Let $\beta \in X_{\rm toy}$ and define 
\begin{align}
R(\beta) := \lim_{s \rightarrow \infty} r_{\beta} .
\end{align}

\begin{lem} $R$  is a smooth function satisfying 
\begin{align}
\frac{\partial R}{\partial \beta} < 0 . 
\end{align} \label{lem:Rb}
\end{lem}

Observe that the order in Lemma \ref{lem:varrho} is reverse to the order in Lemma \ref{lem:Rb}.

\begin{proof}[Proof of Theorem \ref{theo:toptoy}] From  Lemma \ref{lem:AB} it follows that $X_{\rm toy}$ is closed and by Theorem \ref{theo:maintoy} it follows that $X_{\rm toy}$ corresponds to toy steady tip growth solutions. Suppose that ${\rm int}(X_{\rm toy}) \neq \emptyset$. Then there exists an open interval $X_0 \subset X_{\rm toy}$. Take $\beta_1,\beta_2 \in X_0$ with $\beta_1 < \beta_2$. Then, from Lemma \ref{lem:varrho} it follows that $R({\beta_1}) < R({\beta_2})$, see Figure \ref{fig:order}. This is in contradiction with Lemma \ref{lem:Rb}.
\end{proof}

\begin{figure}[h]
\begin{center}
\includegraphics[width=8cm]{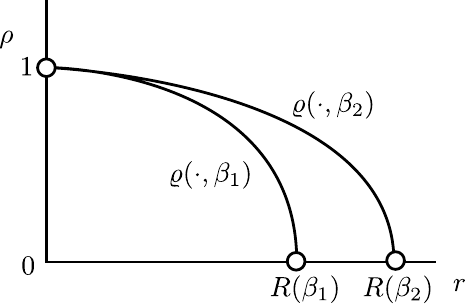}
\caption{Ordering of tip solutions. The white dots correspond to limit points. If there exists an open interval $X_0 \subset X_{\rm toy}$ then for $\beta_1,\beta_2 \in X_0$ with $\beta_1 < \beta_2$ we have that $R({\beta_1}) < R({\beta_2})$.  \label{fig:order}}
\end{center} \end{figure} 

\newpage

\section{The BATS model for fungal tip growth}
\label{secgov}

In this section we review the biological model on which the toy model is based; details can be found in \cite{JONBIO}. This model is called the Ballistic Ageing Thin viscous Sheet (BATS) model. The BATS model gives a description of continuous tip growth in fungal filaments called hyphae. For a short biomechanical overview of the BATS model we refer to \cite{JONCOMFOS,JONIEEE}. For more details concerning the biology of hypha growth we refer to  \cite{CO06,GO17,HA02,KEIJ09,MO08,STE07}.  In Section \ref{sec:toyandbio} the toy model is connected to the BATS model.

Tip growth is a growth stage of a biological cell. During tip growth the cell exhibits extreme lengthwise growth while its shape remains qualitatively the same and the cell's tip velocity remains approximately constant, see Figure \ref{fig:tipgrowth}. Tip growth occurs in a variety of different biological cells, such as fungal filaments, plant root hairs, and flower pollen tubes \cite{GEIT01,GI55}.

Modelling tip growth consists of two aspects: transport of cell wall building material to the cell wall and growth of the cell wall under absorption of cell wall building material. The BATS model relies on an assumption of Bartnicki-Garcia et al. \cite{BG89,BG01} to model cell wall building material transport and an assumption of  Camp\`{a}s and Mahadevan \cite{CM09} to model the growth of the cell wall under absorption of the cell wall building material. Furthermore, a novel equation which models the hardening of the cell wall as it ages is included to derive the BATS model.


\subsection{Modelling tip growth}
\label{modtipgg}

The shape of the cell during tip growth was mathematically described in Section \ref{sec:matdes}. During tip growth the cell grows with constant speed in the direction normal to the tip. In addition, the cell preserves its overall shape. Then, in the $(z,r)$-plane the moving profile at time $t$ is characterised by $(z(s)+ct,r(s))$ where $c$ is the velocity of the tip. We will take $c<0$. To remove the time variable we consider a moving reference frame in which the tip of the cell is fixed at $(z_0,0)$ in the $(z,r)$-plane with $z_0<0$.

At a fixed distance from the cell's tip there is an organelle which transports cell wall building packages, called vesicles, to the cell wall \cite{GI55,GI69}.  Following the work of Bartnicki-Garcia et al.\ \cite{BG89,BG01} it is assumed that vesicles are sent in straight trajectories from an isotropic point source. This point source is called the ballistic Vesicle Supply Center (VSC), see Figure \ref{fig:modeli}. We fix the ballistic VSC at $(z,r)=(0,0)$.

\begin{figure}[h]
\begin{center}
\includegraphics[width=10cm]{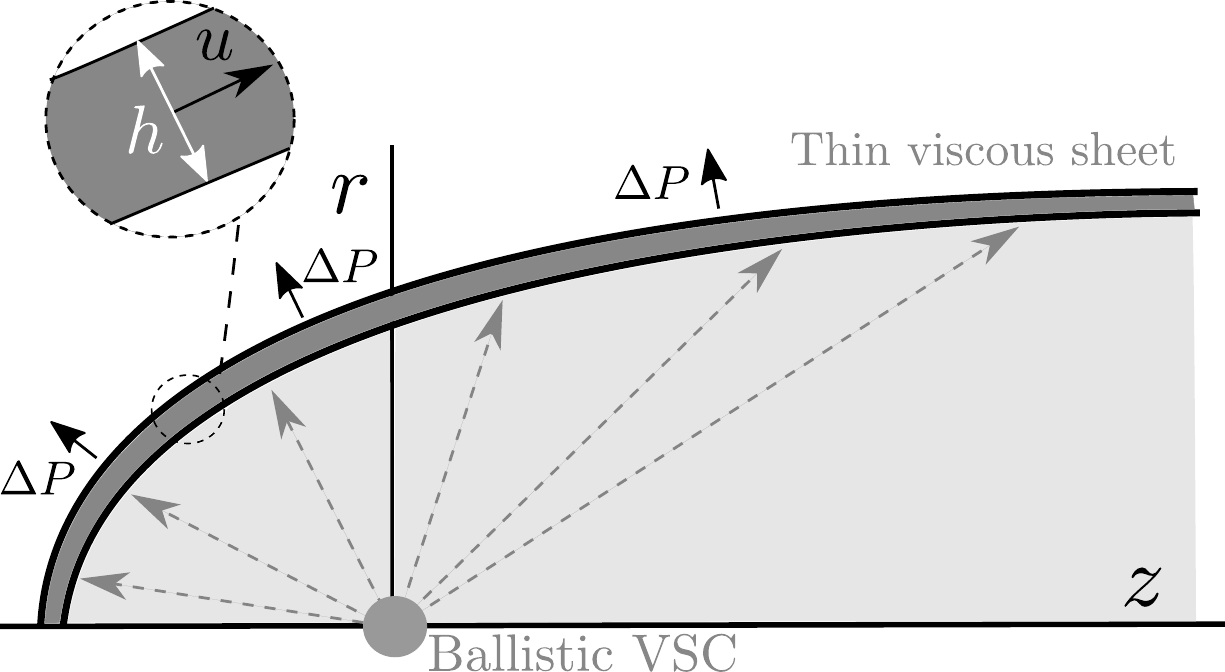}
\caption{Overview of the  BATS model. The cell wall as a thin viscous sheet with ballistic Vesicle Supply Center (VSC).  Vesicles are transported in straight lines to the cell wall from an istropic point source called the ballistic vesicle supply center.  The ballistic vesicle supply center is fixed at $(z,r)=(0,0)$.   The cell wall is modelled as a thin viscous sheet which sustained by a pressure difference $\Delta P$ which exerts a force in the outward normal. The thickness of the cell wall is denoted by $h$ and the tangential velocity of the viscous fluid in the cell wall is denoted by $u$.\label{fig:modeli}}
\end{center} \end{figure}

Following the work of Camp\`{a}s and Mahadevan \cite{CM09} it is assumed that the cell wall is a thin viscous sheet.  This sheet is sustained by a pressure difference $\Delta P$ which is the result from the high pressurised environment in the cell and the comparatively low atmospheric pressure outside the cell. The pressure difference $\Delta P$ generates a force in the direction of the outward normal of the sheet. The thin viscous sheet modelling assumption introduces two new $s$-dependent variables: the thickness of the cell wall denoted by $h$ and the tangential velocity of the cell wall particles denoted by $u$. See Figure \ref{fig:modeli} for an overview of the thin viscous sheet cell wall. Observe that the velocity of the tip can be retrieved by computing $\lim_{s \rightarrow \infty} u(s)$.

To withstand the pressure difference the cell wall is generally strong and rigid. But to ensure rapid growth the tip of the cell wall deforms easier than the cell wall away from the tip, \cite{WES83}. To take this effect into account the BATS model assumes that the cell wall hardens as it ages. Age is a cell wall particle specific value. At arc length $s$ the cell wall has thickness $h(s)$. Hence, the $s$-dependent average variable $\Psi$ is introduced. To compute $\Psi(s)$, an integral equation needs to be solved. The integral equation will be formulated as a differential equation in the next section. To model the hardening of the cell wall as it ages it is assumed that the viscosity depends on $\Psi$: the viscosity at $s$ is given by $\mu(\Psi(s))$, where $\mu \in C^{\infty}(\mathbb{R}_{>0})$ is the viscosity function. Hardening of the cell wall means that the viscosity increases with age. Thus, we require that: 
\begin{align}
\frac{d\mu}{d\Psi}>0.  \label{conbio2}
\end{align}

From an application perspective $\mu$ is a fungus cell specific function since cells can have different material properties based on the species and the cell's physical circumstances.  Hence, $\mu$ is chosen as general as possible in a theoretical setting.


\subsection{Governing equations: 5-dimensional first order ODE}

The BATS model can be expressed as two force balance equations, a mass conservation equation, an age equation and the shape equation (\ref{dzdsarc}). The mass conservation equation is used to eliminate the $u$-variable \cite{JONBIO}. All physical parameters can either be scaled away or absorbed in $\mu$. The non-dimensionalised governing equations can be expressed as the following 5-dimensional first-order ODE: 

\begin{gather}
\begin{aligned}
\rho' &= \frac{3}{2} \frac{(1-\rho)^2}{r} \left(-1+  \frac{ \mu(\Psi) \Gamma(r,z) \rho  \sqrt{1-\rho^2}}{   r^3 }  \right) , \\
r' &= \rho , \\
 h'& = \Big( \frac{r \gamma(\rho,r,z)}{\Gamma(r,z)}- \frac{\rho}{2r} - \frac{ r^2   }{2\mu(\Psi) \Gamma(r,z)  \sqrt{1-\rho^2}}\Big) h , \\
 \Psi' &= \frac{ rh}{\Gamma(r,z)} - \frac{  r \gamma(\rho,r, z) }{\Gamma(r,z) } \Psi,\\
 z' &= \sqrt{1- \rho^2},
\end{aligned}\label{finalfull}
 \end{gather} 
where
\begin{gather}
\begin{aligned}
\gamma(\rho, r,z) &=   
  \frac{r  \sqrt{1-\rho^2}-z  \rho}{(z^2+r^2)^{3/2}},\\
   \Gamma(r,z) &=   1 + \frac{z}{\sqrt{r^2+z^2}} . 
\end{aligned} \label{fluxes}
\end{gather}
We will consider equation \eqref{finalfull} on the phase space given by 
\begin{equation}
M:= \{ (\rho , r,  h , \Psi, z) \in  (-1,1) \times \mathbb{R}_{>0} \times \mathbb{R}_{>0} \times \mathbb{R}_{>0}  \times \mathbb{R}     \}. \label{M}
\end{equation}
The dynamical system corresponding to (\ref{finalfull}) is referred to as the \textit{BATS model}. We define:
\begin{align}
\mathcal{F}:= \{ \mu  \in C^\infty(\mathbb{R}_{>0}) : \; \mu' >0 ,   \lim_{\Psi \rightarrow \infty}\mu(\Psi)= \infty   \} . \label{viscocondi}
\end{align}
For the BATS model (\ref{finalfull}) we let $\mu \in \mathcal{F}$. In \cite{JONBIO} it is shown that $\lim_{\Psi \rightarrow \infty}\mu(\Psi) = \infty$ is a necessary condition for the BATS model \eqref{finalfull} to have solutions which resemble fungal tip growth. 

A solution of \eqref{finalfull}  will be denoted by the vector $x=(\rho ,r , h, \Psi,z)$. When necessary we indicate the dependence of $x$ on $\mu$ by writing $x(\, \cdot \, ; \mu)$.


\subsection{Steady tip growth solutions}
\label{defstgG}

In Section \ref{sec:matdes} a description was given of the cell shape during tip growth in terms of $\rho,r$. Besides the $\rho,r$ variables the BATS model has the variables $h,\Psi,z$. In this section we extend the conditions S1-S4, Section \ref{sec:matdes}, for the toy model to the BATS model.

\begin{itemize}

\item[T1] \textbf{Tip limits:} 
\begin{align*}
\lim_{s \rightarrow 0} \rho(s) &=1,  &   \lim_{s \rightarrow 0} r(s) &=0 ,\\
  \lim_{s \rightarrow 0} h(s) &=h_0 >0,  & \lim_{s \rightarrow 0} \Psi(s) &= h_0 z_0^2, \\
    \lim_{s \rightarrow 0}  z(s)&= z_0 <0 . &    
\end{align*}
\textit{Heuristic explanation:} The $\rho,r$ limits are the same as S1. The limits for  $h, z$ follow directly from the cell shape in Figure \ref{fig:cellshape} and Figure \ref{fig:modeli}. To arrive at the BATS model \eqref{finalfull} the $\Psi$-integral equation was replaced by a differential equation. Computing the limit $s \rightarrow 0$ for this integral equation yields the $\Psi$-limit \cite{JONBIO}. Differently from S1, the limit $\lim_{s \rightarrow 0 }\sqrt{1-\rho(s)^2}/r(s)$ does not appear. This limit is naturally satisfied by the BATS model \eqref{finalfull}, see Appendix \ref{sec:BATSC} for the proof.

\item[T2] \textbf{Analyticity in $r^2$:} There exists $s_0>0$ and $G \in C^{\omega}\left( (-a ,a) , \mathbb{R}^4 \right)$ with $a=r(s_0)^2$ such that
\begin{align*}
(\rho,h,\Psi,z)(s)= G(r(s)^2) \qquad \forall s \in (0,s_0).
\end{align*}
\textit{Heuristic explanation:} Besides S2 we also expect that $h,\Psi,z$ are even functions close to the tip when parametrized in $r$ due to axial symmetry and smoothness at the tip.

\item[T3] \textbf{Global constraints:} For all $s \in \mathbb{R}_{>0}$ the following constraints are satisfied:
\begin{align*}
\rho'(s) <0, \quad \rho(s) >0.
\end{align*}
\textit{Heuristic explanation:}  This condition is identical to S3. 

\item[T4] \textbf{Base limits:}  
\begin{align*}
\lim_{s \rightarrow \infty} \rho(s) &=0,  &   \lim_{s \rightarrow \infty} r(s) &=r_\infty >0,\\
  \lim_{s \rightarrow \infty} h(s) &=h_\infty >0, & \lim_{s \rightarrow \infty} \Psi(s) &= \infty, \\
    \lim_{s \rightarrow \infty}  z(s)&= \infty . &    
\end{align*}
\textit{Heuristic explanation:} Besides S4 we expect that the cell wall thickness converges to a positive constant. The $\Psi$-limit follows from the integral equation that was used to arrive at the BATS model \eqref{finalfull}. Computing the limit $s \rightarrow 0$ for this integral equation yields the $\Psi$-limit \cite{JONBIO}. The $z$-limit follows from the assumption in Section \ref{modtipgg} that the cell's length is infinite.  
\end{itemize}

Solutions of the BATS model which correspond to fungal tip growth are called steady tip growth solutions:

\begin{defi}[Steady tip growth solution] $x=(\rho ,r , h, \Psi,z)$ is a \emph{steady tip growth solution} if it is a solution of the BATS model \eqref{finalfull} that satisfies conditions T1--T4. \label{defi:steadytipgrowth}
\end{defi}

Solutions which are not steady tip growth solutions are not meaningful from a biological perspective. Therefore, proving the existence of steady tip growth solutions is a necessary step in validating the BATS model.

\begin{rem}
It can be shown that condition T4 without $\lim_{s \rightarrow \infty} \Psi(s)=\infty$ implies that $\lim_{s \rightarrow \infty} \Psi(s)=\infty$ \cite{JONBIO}. Hence, in \cite{JONBIO} the condition T4 is formulated without the requirement $\lim_{s \rightarrow \infty} \Psi(s)=\infty$. For convenience we have included the limit in T4. In \cite{JONIEEE} T4 was also formulated with this limit included. 
\end{rem}

\section{Connecting toy model to BATS model\label{sec:toyandbio}}
We will give an analytical derivation of the toy model from the BATS model of Section \ref{secgov}. Then, using the toy model we formulate three conjectures which imply the existence of steady tip growth solutions as given by Definition \ref{defi:steadytipgrowth}.


\subsection{Analytical derivation toy model}

In this section we will derive the toy model \eqref{toygov} from the BATS model \eqref{finalfull}.

Recall that the toy model \eqref{toygov} is a first-order equation with dependent variables $\rho$ and $r$. The $r$-equation in \eqref{finalfull} is given by $r'=\rho$.  Observe that the $\rho$-equation in \eqref{finalfull} has a dependency on $\rho$, $r$, $\Psi$, and $z$. The toy model will be derived by substituting terms in the $\rho$-equation such that the resulting equations only depend on $\rho$ and $r$. These substitutions will be done in such a way that certain asymptotic and global properties are preserved. These substitutions introduce three parameters. Since the bifurcation in Figure \ref{fig:bifu_sol} only requires a single parameter, we will reduce the system to a single parameter.

  
\subsubsection{Substitutions for $z,\Psi$-components}
\label{sec:sub}

To derive the toy model we will substitute $\Gamma(r,z)$ and $\mu(\Psi)$ in the $\rho$-equation of \eqref{finalfull} by $(\rho,r)$-dependent terms which have similar limiting dynamics  for $s \rightarrow s_0$, $s \rightarrow \infty$ and similar global dynamics for $s \in (s_0, \infty)$.   

 Let $x_*=(\rho_*,h_*,\Psi_*,z_*,r_*)$ be a steady tip growth solution as given by Definition \ref{defi:steadytipgrowth}.  We introduce the following substitutions:

\paragraph{Substitution for $\Gamma(r,z)$:} From conditions T1, T3, and T4 we obtain the following:
\begin{align}
\lim_{s \rightarrow s_0}\frac{\Gamma(r_*(s),z_*(s))}{(r_*(s))^2}&= \frac{1}{2z_0^2},& \lim_{s \rightarrow \infty} \frac{\Gamma(r_*(s),z_*(s))}{(r_*(s))^2}&=  \frac{2}{r_\infty^2}, \label{Gammarc} \\ 
\frac{\Gamma(r_*(s),z_*(s))}{(r_*(s))^2}&>0, \qquad \forall s \in \mathbb{R}_{>0}  . \label{Gammarc2}
\end{align} 
Observe that $z_0$ appears in \eqref{Gammarc}. Hence,  a substitution for $\Gamma(r,z)$ should be parameter dependent.  The simplest substitution for $\Gamma(r_*,z_*)$ such that the limits \eqref{Gammarc} exist and the inequality \eqref{Gammarc2} is satisfied is $\alpha_1 r_*^2$ with parameter $\alpha_1  \in \mathbb{R}_{>0}$.

\paragraph{Substitution for $\mu(\Psi)$:} The governing ODE \eqref{finalfull} depends on the viscosity function $\mu$. Hence, the substitution for $\mu(\Psi_*)$ also has a function dependency which will be the function $g$.  Observe that condition T1 and T3 imply that 
\begin{gather} 
\begin{aligned}
\lim_{s \rightarrow s_0} \mu(\Psi_*(s)) \rho_*(s)  &= c_1>0, \\
\mu(\Psi_*(s)) \rho_*(s)   &>0, \qquad \forall s \in \mathbb{R}_{>0} . 
\end{aligned} \label{mumu}
\end{gather}
The base limit condition T4 does not imply that the limit of $\mu(\Psi_*(s)) \rho_*(s)$ for $s \rightarrow \infty$ exists. We would expect that
\begin{align}
\lim_{s \rightarrow \infty } \rho_*'(s) =0. \label{rholimsinff}
\end{align}
The limit \eqref{rholimsinff} together with condition T4 gives
\begin{gather}
\begin{aligned}
\lim_{s \rightarrow \infty}\mu(\Psi_*(s)) \rho_*(s)   =   c_2>0.
\end{aligned} \label{muww}
\end{gather}
Observe that $c_1$ in \eqref{mumu} depends on $h_0$ and $z_0$. Hence, the substitution for $\mu(\Psi_*)$ should be parameter dependent. From condition T2 and the analyticity of $\mu$ it follows that there exists $s_1>s_0$ and $G_1 \in C^{\omega}((-a,a), \mathbb{R})$ with $a = r_{*}(s_1)^2$ such that
\begin{align}
 \mu(\Psi_*(s)) = G_1(r_*(s)^2)   \qquad \forall s \in (s_0,s_1) .  \label{mupmup}
\end{align} 
In other words, we require that the substitution for $\mu(\Psi_*)$ can be written as an analytic function in $r_*^2$. If we substitute $\mu(\Psi_*)$ by 
\begin{align}
  \beta_1 r^2_* \rho^{-1}_* g(r^2_*)  + \beta_2, \;\; \beta_1, \beta_2 \in \mathbb{R}_{>0}, \label{babababa}
\end{align}
 in \eqref{mumu} and \eqref{muww} then the limit exists and the inequality is satisfied. Furthermore, if   $\mu(\Psi_*)$  in \eqref{mupmup}  is substituted  by \eqref{babababa} then there exists a $s_1>s_0$ and a $G_1 \in C^{\omega}((-a,a), \mathbb{R})$ with $a = r_{*}(s_1)^2$  since $\rho_*$ satisfies T2.


\subsubsection{One parameter ODE}

Using the substitutions from Section \ref{sec:sub} we observe that the $(\rho,r)$-equations of \eqref{finalfull} decouple. Observe that the substitutions from Section \ref{sec:sub} introduce three parameters: $\alpha_1, \beta_1, \beta_2$.  We are dealing with a codimension-1 bifurcation. Hence, we reduce the problem to one parameter. Let $\alpha_1 = \beta_2^{-1}$, define $\beta :=\alpha_1^2$, and consider the scaled variables $\tilde{r} =  r /(\alpha_1 \beta_2)$ and $\tilde{s} =  s /(\alpha_1 \beta_2)$. Dropping the tildes gives the toy model \eqref{toygov}.


\subsection{Conjectures}

The existence of toy steady tip growth solutions as given by Definition \ref{def:toy}, Theorem \ref{theo:maintoy}, relies on Corollary \ref{cor:toytip}, Lemma \ref{lem:X} and Lemma \ref{lem:AB}. We re-formulate these in the setting of the BATS model \eqref{finalfull}. This will yield three conjectures. We will use the numerical results obtained in \cite{JONIEEE} to supply evidence which supports these conjectures.


\subsubsection{Tip solutions}

We define solutions which satisfy the location conditions at the tip as given by Definition \ref{defi:steadytipgrowth}. These are an analogue of toy tip solution as given by Definition \ref{def:toytip}. 

\begin{defi}[Tip solutions] $x= (\rho,r,h,\Psi,z)$ is a \emph{tip solution} if it is a solution of the BATS model \eqref{finalfull} which satisfies T1 and T2, and if there exists $s_0 >0$ such that
\begin{align}
\rho_*(s) >0 , \quad \rho'_*(s)<0 \quad \forall s \in (0, s_0).
\end{align} 
\label{defi:tip}
\end{defi}
If follows from Definition \ref{defi:steadytipgrowth} and Definition \ref{defi:tip} that
\begin{align}
\text{steady tip growth solutions } \subset \text{ tip solutions}. \label{subset}
\end{align}
Steady tip growth solutions given in Definition \ref{defi:steadytipgrowth} should occur as the result of a bifurcation of tip solutions as given in Definition \ref{defi:tip}. Hence, the property given in \eqref{subset} is necessary. Observe that \eqref{subset2} is a toy version  \eqref{subset}.

If $x(\, \cdot \, ; \mu) = (\rho,r,h,\Psi,z)(\, \cdot \,)$ is a tip solution then
\begin{align}
\lim_{s \rightarrow s_0} h(s) = h_0, \qquad \lim_{s \rightarrow s_0} z(s) = z_0. \label{h0z0}
\end{align}
Denote the tip solution satisfying \eqref{h0z0} with $\alpha= (h_0,z_0)$ by  $x_\alpha(\, \cdot \, ; \mu)$.  
The asymptotic expansions for tip solutions computed in \cite{JONIEEE} suggest the following:

\begin{con}There exists an open set $Y_\mu  \subseteq \mathbb{R}_{>0} \times \mathbb{R}_{<0}$ such that for all $\alpha \in Y_\mu$ there exists a unique $x_\alpha(\, \cdot \, ; \mu)$. In addition, $F_\mu: \alpha \mapsto x_\alpha (\, \cdot \, ; \mu) $ is continuous with respect to $\alpha$.
\label{con:xa}
\end{con}

Observe that Corollary \ref{cor:toytip} is the toy version of Conjecture \ref{con:xa}. Note that the BATS model \eqref{finalfull} has no parameters. We view $\alpha$ as a parameter dependency in $x_{\alpha}(\, \cdot \, ; \mu )$.  Conjecture \ref{con:xa} implies that there exists a two parameter family of tip solutions.  We note that the numerical results in \cite{JONIEEE} suggest that there exists a viscosity function $\tilde{\mu}$ such that in Conjecture \ref{con:xa} the set $Y_{\tilde{\mu}} = \mathbb{R}_{>0} \times \mathbb{R}_{<0}$.  There also exists a viscosity function $\hat{\mu}$ such that $Y_{\hat{\mu}} \neq \mathbb{R}_{>0} \times \mathbb{R}_{<0}$.  


\subsubsection{Classification of tip solutions}

The numerical results suggest that there exists an open $Y\subset Y_\mu$ such that for all $\alpha \in Y$ the solutions $x_\alpha(\, \cdot \, ; \mu)$ which are \emph{not} steady tip growth solutions (Definition \ref{defi:steadytipgrowth}) are classified by:
\begin{gather}
\begin{aligned}
A_\mu(Y)&:= \{  \alpha  \in Y :   \exists  s_0 \in \mathbb{R}_{>0}, \;  \rho_\alpha(s;\mu) \rho_\alpha'(s;\mu) <0  , \;  \forall s \in (0,s_0),   \rho_\alpha(s_0;\mu) =0       \} ,\\
B_\mu(Y) &:= \{  \alpha \in Y  :   \exists  s_0 \in \mathbb{R}_{>0}, \;   \rho_\alpha(s;\mu) \rho_\alpha'(s;\mu) <0, \;  \forall s \in (0,s_0),   \rho_\alpha'(s_0;\mu) =0   \}. 
\end{aligned} \label{AenB}
\end{gather}

Observe that $A_\mu(Y)$ and $B_{\mu}(Y)$ are the BATS model analogue of $A_{\rm toy}$ and $B_{\rm toy}$ from \eqref{ABtoy}, respectively.  Hence, tip solutions $x_\alpha(\, \cdot \, ; \mu)$  for all $\alpha  \in A_\mu(Y)$ resemble Figure \ref{fig:Amu} and tip solutions $x_\alpha(\, \cdot \, ; \mu)$ for all $\alpha  \in B_\mu(Y)$ resemble Figure \ref{fig:Bmu}. 

We define 
\begin{align}
X_\mu(Y) := Y  \setminus (A_\mu(Y) \cup B_\mu(Y)). \label{XmuY}
\end{align}
Observe that $X_\mu(Y)$ is the  ODE \eqref{finalfull} analogue of $X_{\rm toy}$ from \eqref{Xtoy}. Hence, tip solutions $x_\alpha(\, \cdot \, ; \mu)$  for all $\alpha  \in X_\mu(Y)$ satisfy resemble Figure \ref{fig:Xmu}.

Observe that the definition of the sets $A_\mu(Y)$ and $B_\mu(Y)$ in \eqref{AenB} gives no information on the non-emptiness of $X_\mu(Y)$. If steady tip growth solutions exist, then they must correspond to parameters in $X_\mu(Y)$. Observe that if a $x_\alpha(\, \cdot \, ; \mu)$ satisfies T3, then $\alpha \in X_\mu(Y)$, see Figure \ref{fig:Xmu}. The numerical results in \cite{JONIEEE} suggest the following:

\begin{con}
Let $X_\mu(Y) \neq \emptyset$ with $Y \subset Y_\mu$. If $\alpha \in X_\mu(Y)$ then $x_\alpha(\, \cdot \, ; \mu)$ is a steady tip growth solution as specified in Definition \ref{defi:steadytipgrowth}.
\label{con:X}
\end{con} 

Observe that Lemma \ref{con:X} is a toy version of Conjecture \ref{con:X}.


\subsubsection{Bifurcation diagrams}

In \cite{JONIEEE} the sets $A_\mu(Y),B_\mu(Y)$  have been numerically approximated for a variety of viscosity functions $\mu$. Given a viscosity function $\mu$ we observe one of the following three cases:
\begin{enumerate}
\item $A_\mu(Y) \neq \emptyset, \; B_\mu(Y) = \emptyset$, 
\item $A_\mu(Y) = \emptyset, \; B_\mu(Y) \neq \emptyset$,
\item $A_\mu(Y) \neq \emptyset, \; B_\mu(Y) \neq  \emptyset$.
\end{enumerate}
Only in case 3 the dynamics changes in a qualitative way as the parameter is varied. Hence, the resulting figures can be interpreted as bifurcation diagrams.  For cases 1 and 2 the numerics suggests that $X_\mu(Y) = \emptyset$.  

In Figure \ref{fig:Xbifu} a schematic of case 3 is given. It suggests that if $X_\mu(Y) \neq \emptyset $ then $X_\mu(Y)$ is described by a 1-dimensional smooth curve.

\begin{figure}[h]
\begin{center}
\includegraphics[width=5cm]{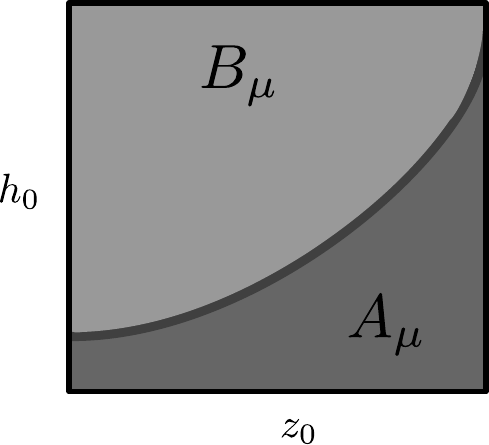}
\caption{Tip solution bifurcation diagram. The square domain corresponds to $Y$. The domain $A_\mu(Y)$ and $B_\mu(Y)$ have been indicated with two different shades of grey. If steady tip growth solutions exist then they must correspond to tip solutions with parameters in $X_\mu(Y) := Y  \setminus (A_\mu(Y) \cup B_\mu(Y))$ which is the curve colored in the darkest grey. \label{fig:Xbifu}}
\end{center} \end{figure} 

Assuming that $X_\mu(Y) \neq \emptyset$  the numerical work suggests that $X_\mu(Y)$ is a 1-dimensional family of bifurcation points since it is located on the boundary of both $A_\mu(Y)$ and $B_\mu(Y)$, see Figure \ref{fig:Xbifu}. Hence, we are dealing with a codimension-1 bifurcation. The definition of $A_\mu(Y)$ and $B_\mu(Y)$ in \eqref{AenB} gives no information on the topology of $A_\mu(Y)$ and $B_\mu(Y)$. Consequently, there is no evidence  that in Conjecture \ref{con:X} the condition  $X_{\mu} \neq \emptyset$ is satisfied. For standard bifurcations, such as a saddle node or Hopf bifurcations, it is straightforward to determine the bifurcation point from a local study. In contrast, the set $X_\mu(Y)$ cannot be determined using a local study since the sets $A_\mu(Y)$ and $B_\mu(Y)$ are not characterized by local behaviour.


\subsubsection{Existence steady tip growth solutions}
\label{sec:exist}

To apply Conjecture \ref{con:X} we require that $X_\mu(Y)\neq \emptyset$. Observe that Figure \ref{fig:Xbifu} gives no evidence that $X_\mu(Y) \neq \emptyset$ since the topology of $A_{\mu}(Y)$ and $B_{\mu}(Y)$ is unknown. Based on Lemma \ref{lem:AB} we formulate the final conjecture:

\begin{con} There exists a connected set $Y \subset Y_\mu$ such that $A_\mu(Y)$ and $B_\mu(Y)$ are non-empty, open, and disjoint sets.
\label{con:AB}
\end{con}

\begin{theo}[Existence of steady tip growth solutions] Let Conjecture \ref{con:xa} be true and suppose that Conjecture \ref{con:X} and Conjecture \ref{con:AB} are true for the same $Y \subset Y_\mu$. Then, a steady tip growth solution as given by Definition \ref{defi:steadytipgrowth} exists. \label{theo:full}
\end{theo}

\begin{proof}
Conjecture \ref{con:AB} implies that $X_\mu(Y) \neq \emptyset$. The theorem then follows from Conjectures \ref{con:xa} and \ref{con:X}. 
\end{proof}


\section{Discussion}
\label{sec:disc}

To prove the existence of steady tip growth solutions as specified by Definition \ref{defi:steadytipgrowth} the Conjectures \ref{con:xa}, \ref{con:X}, and \ref{con:AB} need to proven. The proofs for the toy model give an insight in how to prove these conjectures:

\begin{itemize}
\item[-] \textbf{Conjecture \ref{con:xa}:}  As with the toy model the $\rho$-equation of \eqref{finalfull} is not defined for $(\rho,r)=(1,0)$. In Appendix \ref{sec:exuntoytip} a transformation is introduced which reduces the existence of toy tip solution, Definition \ref{def:toytip}, to the existence of an unstable manifold. It is expected that the transformation can be used to prove the existence of the BATS model's tip solutions, Definition \ref{defi:tip}, by reducing the proof to the existence of an unstable manifold. Asymptotic analysis performed in \cite{JONIEEE} reveals that $Y_\mu$ is dependent on $\mu$. This is different from the toy model \eqref{toygov} since toy tip solutions, Definition \ref{def:toytip}, exist for all $\beta$, Corollary \ref{cor:toytip}.
\item[-] \textbf{Conjecture \ref{con:X}:} Let $\alpha \in X_\mu(Y)$. If Conjecture \ref{con:xa} can be proven then it follows that $x_{\alpha}(\, \cdot \, ; \mu)$ satisfies T1,T2 from Definition \ref{defi:steadytipgrowth}. If the maximal existence interval of $x_{\alpha}(\, \cdot \, ; \mu)$ is $\mathbb{R}_{>0}$ then it will also satisfy T3. The maximal existence interval depends on all the variables. Study of $\rho,r$ as in the toy model is insufficient. However, it is straightforward to prove that the maximal existence interval is $\mathbb{R}_{>0}$, see Lemma \ref{ap:lem:X} in Appendix \ref{ap}. It remains to show T4. For every $\beta$ the toy model \eqref{toygov}  has a unique base limit, S4 in Definition \ref{def:toytip}. This is not the case for the BATS model \eqref{finalfull}. Hence, it is unclear how to prove that T4 of Definition \ref{defi:steadytipgrowth} is satisfied. 
\item[-] \textbf{Conjecture \ref{con:AB}:} Using the same method as in the proof of Lemma \ref{lem:AB} in Appendix \ref{p:lem:AB} we can show the disjointness of $A_{\mu}(Y)$, $B_{\mu}(Y)$ and the openness of $A_{\mu}(Y)$, see Lemma \ref{ap:lem:AB} in Appendix \ref{ap}. The proof of Lemma \ref{lem:AB} does not succeed in proving the openness of $B_{\mu}(Y)$ due to the dependency on the $h,\Psi,z$-variable. As in the proof of Lemma \ref{lem:AB} in Appendix \ref{p:lem:AB} the non-emptiness of $A_{\mu}(Y)$, $B_{\mu}(Y)$ is likely to be obtained by studying degenerate cases. Specifically, for the BATS model \eqref{finalfull} the equations decouple for $h=0$.
\end{itemize}

In conclusion, the application of the toy model's proof-method to the BATS model has been divided into sub-problems which can be studied in future research.  


\subsection*{Acknowledgement} 
This research was partly funded by a PhD grant of the NWO Cluster ``Nonlinear Dynamics in  Natural Systems''(NDNS+) and by NWO VICI grant 639.033.008.


\newpage

\begin{appendices}

\section{Technical proofs I: Construction of toy tip solutions}
\label{sec:exuntoytip}

In this section we present a theorem which constructs toy tip solutions as defined in Definition \ref{def:toytip}. The existence and uniqueness of toy tip solutions presented in Section \ref{sec:toytip} as Corollary \ref{cor:toytip} will follow from this construction theorem. In addition, we will also obtain the smoothness of $\varrho$, defined in \eqref{varrho}, as a corollary of the construction theorem.

Observe that the vector field corresponding to the toy model \eqref{toygov} is not defined for the tip limits given by S1 in Definition \ref{def:toytip}. We apply a series of transformations to the toy model \eqref{toygov} such that we can apply an equilibrium study to the resulting system and show that there exists a solution on the unstable manifold which satisfies the properties of toy tip solutions. 


\subsection{Change of independent variable}

We first perform a change of independent variable. Since in the definition of toy tip solutions, Definition \ref{def:toy}, the limiting asymptotic is described in terms of the variable $s$ we need to be precise with the variable change. For general theory on singularities see \cite{AR92}.

Let $y:=(\rho,r)$ be a solution of  the toy model (\ref{toygov}) restricted to the phase space $M_1$ defined in \eqref{M1}. Let $y$ be defined on the interval $S \subset \mathbb{R}$. Given a $s_1 \in S$ let  $\tau_y : S \rightarrow \mathbb{R}$ satisfy 
\begin{align}
\frac{d \tau_y}{d s} = \frac{\rho}{r}, \qquad \tau_x(s_1)=0. \label{tautaueq2}
\end{align}
Then $\tau_y$ is a diffeomorphism on its range. We introduce the new independent variable $t$ by  $s=\tau^{-1}_y(t)$. We denote the $t$-dependent variables by a hat. The ODE for the $t$-dependent variables is given by 
\begin{gather}
\begin{aligned}
\frac{d\hat{\rho}}{dt} &= \frac{3}{2} \frac{(1-\hat{\rho}^2)}{\hat{\rho}} \left( -1 + \frac{( \hat{\rho} + \beta \hat{r}^2  g(\hat{r}^2) )\sqrt{1-\hat{\rho}^2}}{\hat{r}} \right), \\
\frac{d \hat{r}}{dt} &= \hat{r}. 
\end{aligned} \label{toygovtau}
\end{gather}
As phase space we take the set $M_1$ as defined in equation \eqref{M1}.

Observe that equation \eqref{tautaueq2} allows us to describe $t$ in terms of the $s$-dependent variables. We also need a transformation to describe $s$ in terms of the $t$-dependent variables since we want to prove uniqueness of toy tip solutions.  Let $\hat{y}:=(\hat{\rho},\hat{r})$ be a solution of \eqref{toygovtau} in the phase space $M_1$ which is defined on the interval $T$. This induces a solution for \eqref{toygov}. More specifically, given a $t_0 \in T$ let $\sigma_{\hat{y}}: T \rightarrow \mathbb{R}$ satisfy 
\begin{align}
\frac{d\sigma_{\hat{y}}}{d t} = \frac{\hat{r}}{\hat{\rho}}, \qquad \sigma_{\hat{y}}(t_0)=0.  \label{toysigma}
\end{align}
Then $\sigma_{\hat{y}}$ is a diffeomorphism on its range and $(\hat{\rho} \circ \sigma^{-1}_{\hat{y}} ,\hat{r} \circ \sigma^{-1}_{\hat{y}} )$ is a solution of the $s$-dependent toy model \eqref{toygov}. 

Observe that the vector field of \eqref{toygovtau} is not defined for $ \hat{r}=0$ which is the S1 tip limit condition on the $r$-variable of the tip solution as defined in Definition \ref{def:toytip}. Hence, we need to introduce new dependent variables.


\subsection{The dependent variables $\eta$ and $w$}
\label{subsec:etaw}

We introduce the new dependent variable ${\eta} = \sqrt{1- \hat{\rho}^2}/\hat{r}$. We will see that the resulting $\eta$-equation is quadratic in $\hat{r}$. Hence, we will introduce the new dependent variable $w=\hat{r}^2$. More specifically, we define
\begin{align*}
N_0  :=  \{ ({\eta} ,w) \in \mathbb{R}_{>0} \times \mathbb{R}_{>0}  \; : \; {\eta}^2 w < 1    \}.
\end{align*}
Then the map $\Phi : M_1 \rightarrow N_0$ given by
\begin{align}
\Phi(\hat{\rho},\hat{r}) = \left(\frac{\sqrt{1- \hat{\rho}^2}}{\hat{r}} , \hat{r}^2 \right ) , \label{defphit}
\end{align}
is a diffeomorphism. We consider the ODE corresponding to the variables $({\eta},w) := \Phi(\hat{\rho},\hat{r})$.  The $({\eta},w)$-ODE is given by: 
\begin{gather}
\begin{aligned}
\frac{d {\eta}}{dt} &=  \frac{{\eta}  }{2 }  \left( 1- 3   {\eta} \sqrt{1- {\eta}^2 w } \right) - \frac{3}{2} \beta   {\eta}^2 w g(w) ,\\
\frac{d w}{dt} &=  2w. 
\end{aligned}  \label{etartau}
\end{gather}
Observe that the vector field corresponding to \eqref{etartau} has two equilibria: $(0,0)$ and $(1/3,0)$. The tip limits condition S1 in Definition \ref{def:toy} requires that the tip limit must satisfy $\eta>0$.  Consequently, we are only interested in the equilibrium $q_0: = (1/3,0)$. Observe that $q_0  \notin N_0$. Hence, we consider the $({\eta},w)$-ODE \eqref{etartau} on the phase space
\begin{align*}
 N_1 := \{ ({\eta} ,w ) \in \mathbb{R}_{>0} \times \mathbb{R}  \; : \; {\eta}^2 w < 1    \}.
\end{align*}
Observe that $q_0 \in N_1$ and that the vector field corresponding to \eqref{etartau} is smooth on $N_1$. Figure \ref{fig:space} shows the relation between the phase spaces $N_0$ and $N_1$.

\begin{figure}[h]
\begin{center}
\includegraphics[width=3cm]{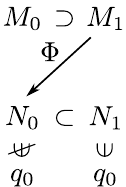}
\caption{Relation between different phase spaces. We considered the toy model \eqref{toygov} on the phase space $M_0$. Then, the phase space was restricted to $M_1$ such that we have the diffeomorphism $\Phi: M_1 \rightarrow N_0$. The phase space $N_0$ contains no equilibria. We then extend the phase space to $N_1$ which contains the equilibrium $q_0 = (1/3, 0)$.  \label{fig:space}}
\end{center} \end{figure} 


\subsection{Toy $t$-tip solutions}
\label{sec:toyttip}

We need to formulate a definition of toy tip solutions in the new variables:

\begin{defi}[Toy $t$-tip solutions] A solution $(\eta,w)$ of the ODE \eqref{etartau} is a \emph{toy $t$-tip solution} if and only if it satisfies: 
\begin{itemize}
\item[$S'1$] \textbf{ Limit $t \rightarrow -\infty$:} 
\begin{align*}
\lim_{t \rightarrow -\infty} \eta(t) = \frac{1}{3}, \qquad \lim_{t \rightarrow -\infty} w(t) =0 .
\end{align*}
\item[$S'2$] \textbf{Analyticity in $w$:} There exists a $t_0\in \mathbb{R}$ and a $G \in C^{\omega}\left( (-a ,a) , \mathbb{R} \right)$ with $a=w(t_0)$ such that
\begin{align*}
\eta(t)= G(w(t)) \qquad \forall t \in (-\infty,t_0).
\end{align*}
\end{itemize}
\label{def:toyttip}
\end{defi}

\begin{lem} Let $\tau_y, \sigma_{\hat{y}}$ and $\Phi$ be given by \eqref{tautaueq2}, \eqref{toysigma} and \eqref{defphit}, respectively. Then we have the following: 
\begin{itemize}
\item[1.] If $y$ is a toy tip solution (Definition \ref{def:toytip}), then $\Phi(y \circ \tau^{-1}_{y})$ is a toy $t$-tip solution (Definition \ref{def:toyttip}),  
\item[2.] If $\hat{y}$ is a toy $t$-tip solution (Definition \ref{def:toyttip}), then $\Phi^{-1}(\hat{y} \circ \sigma^{-1}_{\hat{y}})$ is a toy tip solution (Definition \ref{def:toytip}) up to translation in the independent variable.
\end{itemize} 
\label{lem:toyequiv}
\end{lem}

\begin{proof}
We will first prove statement 1. Let $y=(\rho,r)$ be a toy tip solution as specified by Definition \ref{def:toytip}. We will prove the following:

\qquad \textbf{Claim I:}
\begin{align*}
\lim_{s \rightarrow 0} \Phi(y(s))  = \left(  \frac{1}{3} , 0 \right).
\end{align*}
Since $y$ satisfies S1 of Definition \ref{def:toytip} we only need to prove that
\begin{align*}
\lim_{s \rightarrow 0}  \frac{\sqrt{1-\rho(s)^2}}{r(s)} = \frac{1}{3} .
\end{align*}
From S2 of Definition \ref{def:toytip} it follows that there  a $s_0>0$ and a $G \in C^{\omega}( (-a,a), \mathbb{R}_{>0})$ with $a=r_*(s_0)^2$  such that  
\begin{align*}
\rho(s) = G(r(s)^2) \; \; \; \forall s \in (0,s_0).
\end{align*}
Let $\tilde{\eta}:=\sqrt{1-\rho^2}/r$. Then, it follows from S1, S2 that
\begin{align}
\lim_{s \rightarrow 0}\tilde{\eta}'(s)  =0. \label{tildeetanul}
\end{align}
Using \eqref{etartau} we obtain that
\begin{align}
\tilde{\eta}' =    \frac{{\tilde{\eta}} \rho }{2 r }  \left( 1- 3   {\tilde{\eta}} \sqrt{1- {\tilde{\eta}}^2 r^2 } \right) - \frac{3}{2} \beta  \rho {\tilde{\eta}}^2 r g(r^2)  . \label{tildeetaeq}
\end{align}
Combining \eqref{tildeetanul} and \eqref{tildeetaeq} we get
\begin{align*}
0 &= \lim_{s \rightarrow 0} \frac{{\tilde{\eta}(s)} \rho(s) }{2 r(s) }  \left( 1- 3   {\tilde{\eta}} \sqrt{1- {\tilde{\eta}(s)}^2 r(s)^2 } \right) - \frac{3}{2} \beta  \rho(s) {\tilde{\eta}(s)}^2 r(s) g(r(s)^2)  \\
& \Rightarrow  \lim_{s \rightarrow 0} \tilde{\eta}(s)  = \frac{1}{3} . 
\end{align*}
Hence, we have proven Claim I. 

Observe that from \eqref{tautaueq2} it follows that
\begin{align*}
\lim_{s \rightarrow 0 }\tau_y(s) =  \lim_{s \rightarrow 0}\ln \frac{r(s)}{r(s_1)} = -\infty .
\end{align*}
There exists a $s_2$ such that $\frac{d\tau_y}{dt}(s) >0 $ for all $s \in (0, s_2)$. Consequently, we get that 
\begin{align}
 \lim_{t \rightarrow -\infty }\tau^{-1}_{y}(t) =0 .\label{justtau0}
\end{align}
We then define $\hat{y} :=  \Phi(x \circ \tau^{-1}_{y})$. It follows from Claim I and \eqref{justtau0} that $\hat{y}$ satisfies S$'$1 of Definition \ref{def:toyttip}. Since $\Phi$ is analytic $\hat{y}$ satisfies S$'$2 of Definition \ref{def:toyttip}.

We continue with proving statement 2. Let $\hat{y}=(\eta,w)$ be a $t$-tip solution as specified by Definition \ref{def:toyttip}. We will first prove the following:

\qquad \textbf{Claim II:}
\begin{align*}
\lim_{t \rightarrow - \infty} \Phi^{-1} (\hat{y}(t)) = (1,0). 
\end{align*} 
Using that $\hat{y}$ satisfies S$'$1 we obtain that
\begin{align*}
\lim_{t \rightarrow -\infty} \sqrt{1-w(t) \eta(t)^2} = 1 , \qquad \lim_{t \rightarrow - \infty} w(t) =0
\end{align*}
This proves Claim II. 

Observe that $w$ is a non-zero function. Then, it follows that
\begin{align}
w(t) = c {\rm e}^{2 t},  \label{wclltoy}
\end{align}
with $c$ a positive constant. We consider $\sigma_{\Phi^{-1}(\hat{y})}$. Then using \eqref{wclltoy} it follows that
\begin{align*}
\lim_{t \rightarrow -\infty }\sigma_{\Phi^{-1}(\hat{y})}(t) &=  \lim_{t \rightarrow -\infty } - \int^{t_0}_{t} \frac{\sqrt{w(\tau)}}{\sqrt{1-\eta(\tau)^2 w(\tau)}} d\tau , \\
& = \lim_{t \rightarrow -\infty } - \int^{t_0}_{t} \frac{ \sqrt{c} {\rm e}^{\tau}}{\sqrt{1- c \eta(\tau)^2 {\rm e}^{\tau} }} d\tau .
\end{align*}
Consequently, there exists a constant $s_0$ such that
\begin{align*}
\lim_{t \rightarrow -\infty }\sigma_{\Phi^{-1}(\hat{y})}(t) =s_0 . 
\end{align*}
There exists a $t_1$ such that $ \frac{d \sigma_{\Phi^{-1}(\hat{y})}}{dt}(t) >0$ for all $t \in (-\infty, t_1)$. Consequently, we get that
\begin{align}
\lim_{s \rightarrow s_0} \sigma^{-1}_{\hat{y}}(s) = -\infty . \label{sigmas00}
\end{align} 
We define $y:= \Phi^{-1}(\hat{y} \circ \sigma^{-1}_{\hat{y}})$. We will show that $y(s-s_0)$ is a toy tip solution as defined by Definition \ref{def:toytip}. From Claim II and \eqref{sigmas00} it follows that $y(s-s_0)$ satisfies S1. If follows from S$'$2 that $y(s-s_0)$ satisfies S2. Let $(\rho,r):= y$. It follows from the domain of $\Phi$ that $\rho >0$. Using S$'$1 we get 
\begin{align*}
\lim_{s \rightarrow s_0}  \frac{r(s) \rho'(s)}{1-\rho(s)^2} &= \lim_{s \rightarrow s_0}  \frac{3}{2}  \left( -  1+  \frac{ \sqrt{1-\rho^2}  (   \beta  {r}^2 g(  {r}^2) + \rho    )   }{{r}}    \right) =  -1.
\end{align*}
Hence, by the toy model \eqref{toygov} we obtain that there exists a $s_1$ such that  
\begin{align*}
\rho(s)>0, \;\; \rho'(s)<0 , \; \;\; \forall s \in (s_0,s_1). 
\end{align*}
Consequently, $y(s-s_0)$ satisfies \eqref{localtoytipineq}.
\end{proof}


\subsection{The invariant manifold}

We will show the existence of an invariant manifold which corresponds to toy $t$-tip solutions as specified by Definition \ref{def:toyttip}. We consider the equilibrium $q_0:=(1/3,0)$ of the $({\eta},w)$-ODE \eqref{etartau}. The eigenvalues corresponding $q_0$ are given by
\begin{align}
 \lambda_1 =  - \frac{1}{2}, \qquad \lambda_2 =  2. \label{eigiii}
\end{align}
Consequently, $q_0$ has a 1-dimensional stable manifold and a 1-dimensional unstable manifold.  It follows from S$'$1 of Definition \ref{def:toyttip} and Lemma \ref{lem:toyequiv} that toy tip solutions correspond to the unstable manifold. Denote the unstable manifold by $W_\beta^u(q_0)$ and the corresponding unstable subspace by $E^{u}_\beta$. We have that
\begin{align}
E^{u}_\beta = {\rm span}  \Big(   \frac{1}{18}- \beta  g(0),15)^\top \Big) . \label{unstabsubspace}
\end{align}

To prove the Corollary \ref{cor:toytip} we will need the following:

\begin{lem}
\label{lem:WaGs}
There exists an interval  $( a_\beta,b_\beta)$ with $0 \in ( a_\beta,b_\beta)$ and a unique $G_\beta \in C^{\omega}((a_\beta,b_\beta), \mathbb{R})$ such that  
\begin{align}
W_\beta^u(q_0) =   \{ (G_\beta(w) , w) \in N_1 \; : \; w \in (a_\beta,b_\beta)        \},   \label{WaGs}
\end{align}
$\beta \mapsto G_\beta$ is smooth, $\beta \mapsto a_{\beta}$ is lower semi-continuous and  $\beta \mapsto b_{\beta}$ is upper semi-continuous. 
\end{lem}

\begin{proof}
The existence  $G_\beta \in C^{\omega}((a_\beta,b_\beta), \mathbb{R})$ in \eqref{WaGs} follows from Lemma \ref{lem:mequiv} and S$'$2 in Definition \ref{def:toyttip}. We extend the toy model \eqref{toygov} by adding the equation:
\begin{align*}
\beta ' =  0. 
\end{align*}
Then, the maximal center unstable manifold corresponding to the equilibrium $(q_0,\beta)$ is unique since the toy model \eqref{toygov} has a unique unstable manifold. Since the center unstable manifold is unique it is $C^k$ for all $k \geq 0$ smooth which implies that $\beta \mapsto G_\beta$ is smooth. The lower semi-continuity of $\beta \mapsto a_{\beta}$   and the upper semi-continuity of $\beta \mapsto b_{\beta}$ follow from standard ODE theory, for example see Theorem 6.2 in \cite{SID13}.
\end{proof}

Toy $t$-tip solutions are contained in $N_0 \subset N_1$ and  $W_\beta^u(q_0) \subset N_1$.  Consequently, we consider $W_{\beta}^u(q_0)  \cap N_0$. 

\begin{lem}
\label{connectednone}
$W_{\beta}^u(q_0)  \cap N_0$ is a non-empty connected set.
\end{lem}

\begin{proof}
Let  $ \mathbf{v} \in E^{u}_\beta$ be a non-zero vector. Then $\mathbf{v} \notin  T_{q_0} \partial{N}_0$ and $W_{\beta}^u(q_0)  \cap N_0 \neq \emptyset$.  Observe that $W_{\beta}^u(q_0) \setminus \{q_0 \}  $ does not intersect the curve $w=0$ since $dw/dt \neq 0$ if $w \neq 0$.  Then, since 
\begin{align*}
N_0    = \{  (\eta,w)  \in N_1 \; : \; w>0    \}, 
\end{align*}
it follows that $W_{\beta}^u(q_0)  \cap N_0$ is connected.
\end{proof}

The next lemma states the relation between $W_{\beta}^u(q_0)  \cap N_0$ and toy $t$-tip solutions:

\begin{lem} A solution $(\eta,w) \in C^{\infty}((-\infty,t_0),N_0)$ of the ODE \eqref{etartau} with initial conditions in $W_{\beta}^u(q_0)  \cap N_0$ is a toy $t$-tip solution (Definition \ref{def:toyttip}). 
\label{lem:mequiv}
\end{lem}

\begin{proof}
Let $(\eta,w)$ be the solution given in the lemma. Then it satisfies S$'$1 from Definition \ref{def:toyttip}. It remains to prove S$'$2.  The spanning vector of $E^{u}_\beta$ has non-zero $w$-component. Applying the analytic version of the unstable manifold theorem gives S$'$2.
\end{proof}

We return to the $(\rho,r)$-variables as the $(\eta,w)$-variables are not suitable for the proofs of the lemmas for the main theorems, Theorem \ref{theo:maintoy} and Theorem \ref{theo:toptoy} in Section \ref{sec:toymaintheo}. Consequently, we need to express $W_{\beta}^u(q_0)  \cap N_0$ in terms of the $(\rho,r)$-variables. We define 
\begin{align}
W^{\rm loc}_{ \beta}  : =  \Phi^{-1}( W_{\beta}^u(q_0)  \cap N_0)) . \label{Wlocalpha}
\end{align}
Observe that $W^{\rm loc}_\beta   \subset  M_1 $.  It follows by \eqref{defphit} and Lemma \ref{connectednone} that $W^{\rm loc}_{ \beta} $ is a non-empty connected set. Denote the flow corresponding to \eqref{toygov} by $\phi_\beta$. We can use the flow $\phi_\beta$ to extend  $W^{\rm loc}_\beta$ to $M_0$:   
\begin{align}
{W}^{\rm tip}_{ \beta}:= \{ y_0 \in M_0 \; : \; \exists s_0 \in \mathbb{R} \; {\rm s.t.} \; \phi_\beta(y_0;s_0) \in W^{\rm loc}_\beta  \}. \label{wtipatoy}
\end{align}
Observe that $W^{\rm tip}_\beta$ is the global version of $W^{\rm loc}_\beta$.


\subsection{$W^{\rm tip}_\beta$ are toy tip solutions}
\label{p:cor:toytip}

Toy tip solutions can be constructed using $W^{\rm tip}_\beta$:

\begin{theo}[Construction toy tip solutions] Let $y$ be a solution of the toy model \eqref{toygov} for $\beta \in \mathbb{R}$. Then the following two statements are equivalent:
\begin{itemize}
\item[1.] $y$ is a toy tip solution (Definition \ref{def:toytip}) up to a translation of the independent variable;
\item[2.] ${\rm range}\left(y \right) = {W}^{\rm tip}_{ \beta}$.
\end{itemize}  \label{theo:toytiptheorem}
\end{theo}

\begin{proof}
Let $y$ be a solution of the toy model \eqref{toygov} for $\beta \in \mathbb{R}$. 
$1 \Rightarrow 2$: By Lemma \ref{lem:toyequiv} it follows that $\hat{y} := \Phi(y \circ \tau_{y}^{-1})$ is a $t$-tip solution as defined by Definition \ref{def:toyttip}. Since $\hat{y}$ satisfies S$'$1 in Definition \ref{def:toyttip} and since $q_0$ has a 1-dimensional unstable manifold it follows that  ${\rm range}\left(y \right) = {W}^{\rm tip}_{ \beta}$.
$2 \Rightarrow 1$: Let $\hat{y} := \Phi(y \circ \tau_{y}^{-1})$. It follows from Lemma \ref{lem:mequiv} that $\hat{y}$ is a toy $t$-tip solution. Then using Lemma \ref{lem:toyequiv} we obtain that $y$ is a toy tip solution up to a translation in the independent variable.
\end{proof}

\begin{proof}[Proof of Corollary \ref{cor:toytip}]
The uniqueness of toy tip solutions (Definition \ref{def:toytip}) follows from Theorem \ref{theo:toytiptheorem}. It remains to prove that $F_{\rm toy}: \beta \mapsto (\rho_\beta,r_\beta)$ is continuous in $\beta$. 
We return to the $(\eta,w)$-ODE \eqref{etartau}. Denote the toy $t$-tip solution corresponding to $\beta$ by $(\eta_\beta,w_\beta)$. Then, it follows Lemma \ref{lem:WaGs} and the $w$-equation in \eqref{etartau}  that $\beta \mapsto (\eta_\beta,w_\beta)$ is continuous in $\beta$. Then, transforming the dependent variable to $\rho,r$ using \eqref{defphit} and the independent variable to $s$ using \eqref{toysigma} it follows that $F_{\rm toy}: \beta \mapsto (\rho_\beta,r_\beta)$ is continuous in $\beta$. 
\end{proof}

\begin{cor} The function $\varrho$ defined in \eqref{varrho} is a smooth function satisfying 
\begin{align}
\lim_{r \rightarrow 0 } D \varrho (r,\beta) = \mathbf{0}.  \label{limitrhonul}
\end{align}
\label{cor:varrhosmooth}
\end{cor}

\begin{proof}
The smoothness of $\varrho$ follows by changing ${W}^{\rm tip}_{ \beta}$ to the $(\rho,r)$-variables using \eqref{defphit},  the properties obtained in   Lemma \ref{lem:WaGs} and the equivalence in Theorem \ref{theo:toytiptheorem}. The limit \eqref{limitrhonul} then follows from the analyticity condition S2 in Definition \ref{def:toytip}.
\end{proof}


\section{Technical proofs II: Lemmas for Theorems \ref{theo:maintoy} and \ref{theo:toptoy}}
\label{sec:tech}

In this Section we give the proof of Lemmas \ref{lem:X}--\ref{lem:Rb} which were used to prove Theorems \ref{theo:maintoy} and \ref{theo:toptoy}. These lemmas rely on the construction of toy tip solutions as presented in Theorem \ref{theo:toytiptheorem}. As in Section \ref{sec:maintoy} the toy tip solution corresponding to the parameter $\beta$ will be denoted by $(\rho_\beta,r_\beta)$.


\subsection{Proof of Lemma \ref{lem:X}} 
\label{p:lem:X}

Recall the set $X_{\rm toy}$ defined in \eqref{Xtoy}.  Denote by $I^\beta_{\rm max}$ the maximal existence interval of $(\rho_\beta,r_\beta)$. The proof of Lemma \ref{lem:X} consists of two parts. First, it needs to be shown that if $\beta \in X_{\rm toy}$, then $I^\beta_{\rm max}= \mathbb{R}_{>0}$. Next, it needs to be shown that  $(\rho_\beta,r_\beta)$ satisfies conditions S1--S4 in Definition \ref{def:toy}.

Take $\beta \in X_{\rm toy}$.  We first prove that $I^\beta_{\rm max}= \mathbb{R}_{>0}$. We prove this by contradiction.  Hence, suppose that  $I^\beta_{\rm max}=(0,s_0)$ where $s_0>0$.  Then it follows that $\lim_{s \rightarrow s_0} \rho_\beta(s)  = c_0 \in[0,1)$ and $\lim_{s \rightarrow s_0} r_\beta(s)  = c_1 >0$. Then $(c_0,c_1) \in M_0$ which yields a contradiction. Consequently, we have shown that $I^\beta_{\rm max}= \mathbb{R}_{>0}$. 

We continue with showing that $(\rho_\beta,r_\beta)$ is a toy steady tip growth solution. Observe that $(\rho_\beta,r_\beta)$ satisfies S1, S2, and S3. Hence, we only need to prove that it satisfies S4: $\lim_{s \rightarrow \infty }(\rho_\beta,r_\beta)(s) = (0,r_\infty)$ with $r_\infty >0$. If $(\rho_\beta,r_\beta)$ does not satisfy S4 then $\lim_{s \rightarrow \infty }(\rho_\beta,r_\beta)(s)= (c_0, \infty)$ with $c_0 \in [0,1)$.  Then, it follows by the $\rho$-equation that there exists a $s_1$ such that $\rho_\beta'(s) >0 $ for all $s \in (s_1, \infty)$. This contradicts the definition of $X_{\rm toy}$.


\subsection{Proof of Lemma \ref{lem:AB}}
\label{p:lem:AB}

We need to prove the following three claims:
\begin{description}
\item[Claim 1:] $A_{\rm toy}  \cap  B_{\rm toy} = \emptyset $;
\item[Claim 2:] $A_{\rm toy}$ is non-empty and open;
\item[Claim 3:]  $B_{\rm toy}$ is non-empty and open.
\end{description}

\begin{proof}[Proof of Claim 1]
We will use contradiction. Suppose that $\beta \in A_{\rm toy}  \cap  B_{\rm toy} \neq \emptyset$. Then there exists a $s_0 >0$ such that $\rho_\beta(s) \rho_\beta'(s) <0$ for all $s \in (0,s_0)$ and $\rho_\beta'(s_0)=\rho_\beta(s_0) =0$. It then follows that $\lim_{s \rightarrow s_0} r_{\beta}(s_0) = \infty$ but $r_{\beta}$ is bounded since $|r_\beta'| \leq 1$.
\end{proof}

\begin{proof}[Proof of Claim 2]
We define 
\begin{align*}
\tilde{A}_{\rm toy} &:= \{  \beta  \in  \mathbb{R} \; :  \; \exists  s_0 \in \mathbb{R}_{>0}  \; \;  \rho_\beta(s) \rho_\beta'(s) <0 \; \forall s \in (0,s_0), \;  \rho_\beta(s_0) =0       \} .
\end{align*}
We will prove the following two claims:
\begin{itemize}
\item[\;] Claim I: The set $\tilde{A}_{\rm toy}$ is open,
\item[\;] Claim II: $0 \in \tilde{A}_{\rm toy}$.
\end{itemize}
Observe that Claim I and II imply that $A_{\rm toy}$ is a non-empty open set. We first prove Claim I. Let $\beta \in \tilde{A}_{\rm toy}$. Then there exists a $ s_0 \in \mathbb{R}_{>0}$ such that  $ \rho_\beta(s) \rho_\beta'(s) <0$ for all $s \in (0,s_0)$ and such that $ \rho_\beta(s_0) =0$. It follows from Claim I  that $\rho_\beta'(s_0) < 0$. Consequently, there exists a $s_1>s_0$ such that $\rho_\beta(s)<0$ and $\rho_\beta'(s)<0$ for all $s \in (s_0,s_1)$. It follows from Corollary \ref{cor:toytip} that there exists a $\delta>0$ such that $(\beta- \delta, \beta + \delta) \subset \tilde{A}_{\rm toy}$. Hence, $\tilde{A}_{\rm toy}$ is open. 
We will proceed with the proof of Claim II.  We will use contradiction. Suppose that $0 \notin \tilde{A}_{\rm toy}$. Then one of the following cases is true:
\begin{itemize}
\item[] Case I: $ 0 \in X_{\rm toy}$
\item[] Case II: $0 \in B_{\rm toy}$ 
\end{itemize}
If Case I is true then it follows by Lemma \ref{lem:X} that $(\rho_0,r_0)$ is a steady tip growth solution. It follows from the toy model \eqref{toygov} that $(\rho_0,r_0)$ cannot be a steady tip growth solution since the vector field corresponding to $\beta=0$ is non-zero for $\rho=0$. The nullcline corresponding to the $\rho$-equation for $\beta=0$ is given by
\begin{align*}
L_0 := \left\{ (\rho,r) \in M_0 \; : \;   \frac{\rho \sqrt{1-\rho^2}}{r} = 1  \right\}.
\end{align*}

A qualitative picture of the level set $L_0$ is displayed in Figure \ref{fig:toyisonull}. 
    \begin{figure}[t]
\begin{center}
\includegraphics[width=4cm]{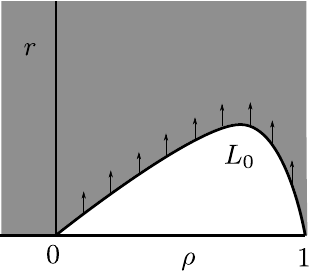}
\caption{The direction of the vector field along the nullcline $L_0$. The grey region corresponds to $\rho'<0$. The white region corresponds to $\rho'<0$.   \label{fig:toyisonull}}\end{center} \end{figure} 
From Figure \ref{fig:toyisonull} it follows that $\rho_0'<0$.    Consequently, Case II does not occur. This completes the contradiction.   
\end{proof}

\begin{proof}[Proof of Claim 3]
We will first assume that $B_{\rm toy}$ is non-empty and prove that $B_{\rm toy}$ is open. Let $\beta \in B_{\rm toy}$. Then there exists a $s_0$ satisfying  $ \rho_\beta(s) \rho_\beta'(s) <0$ for all $s \in (0,s_0)$ and $ \rho'_\beta(s_0) =0$. In addition, we claim that  $ \rho''_\beta(s_0) > 0$. To prove this claim observe that $\rho_\beta'(s)<0$ for all $s \in (0,s_0)$ implies that $ \rho''_\beta(s_0)  \geq  0$. Hence, it remains to show that   $ \rho''_\beta(s_0)  \neq  0$.

We will proceed with a proof by contradiction by assuming that  $ \rho''_\beta(s_0)  =  0$. We compute 
\begin{align*}
\rho_\beta''(s_0) &= \frac{3}{2} \frac{1-\rho_\beta(s_0)^2}{r_\beta(s_0)} \Big( - \frac{\rho_\beta(s_0) (\rho_\beta(s_0) + \beta r_\beta(s_0)^2 g(r_\beta(s_0)^2) \sqrt{1-\rho_\beta(s_0)^2})}{r_\beta(s_0)^2} \\
& \hspace{1cm} + \frac{\beta \rho_\beta(s_0)  }{r_\beta(s_0)}  \frac{\partial r^2 g(r^2)}{\partial r} \Big|_{r=r_\beta(s_0)} \Big)   \\ 
&=  \frac{3}{2} \frac{(1-\rho_\beta(s_0)^2)\rho_\beta(s_0)}{r_\beta(s_0)^2} \left( -1 + \beta \sqrt{1-\rho_\beta(s_0)^2}  \frac{\partial r^2 g(r^2)}{\partial r} \Big|_{r=r_\beta(s_0)}   \right),  \\
\rho'''(s_0) &= \frac{3\beta }{2} \frac{(1-\rho_\beta(s_0)^2)^{3/2}\rho_\beta(s_0)^2}{r_\beta(s_0)^2}  \frac{\partial^2 r^2 g(r^2)}{\partial r^2} \Big|_{r=r_\beta(s_0)} >0.
 \end{align*}
If $\rho_\beta'(s_0) = \rho''_\beta(s_0) =  0$ and $\rho'''(s_0)>0$ then there exists a $s_1<s_0$ such that $\rho'(s)>0$ for all $s \in (s_1,s_0)$ which yields a contradiction. Thus we obtain that $ \rho''_\beta(s_0)  >  0$. Consequently, there exists a $s_2>s_0$ such that  $ \rho'_\beta(s) >0$ , $\rho(s)>0$ for all $s \in (s_0,s_2)$. It follows from Corollary \ref{cor:toytip} that there exists a $\delta>0$ such that $(\beta- \delta, \beta + \delta) \subset B_{\rm toy}$. Hence, $B_{\rm toy}$ is open.

It remains to prove that $B_{\rm toy}$ is non-empty. Denote by $w_\beta(\rho,r)$ the $\rho$-component of the vector field corresponding to the toy model \eqref{toygov}.  Since $g'(r)>0$ for all $r>0$ and  $\lim_{r \rightarrow \infty }g(r) =\infty $ we have that there exists a $\beta_1>0$ such that 
\begin{align*}
w_{\beta_1}(1/2,r) >0 \;  \; \; \forall r>0.
\end{align*}
This implies that $\rho_{\beta_1}(s)>1/2$ for all $s$ in its maximal existence interval. Therefore, we have that  $\beta_1 \in \mathbb{R}_{>0} \setminus ( A_{\rm toy} \cup X_{\rm toy}) = B_{\rm toy} $.
\end{proof}


\subsection{Proof of Lemma \ref{lem:varrho}}
\label{p:lem:varrho}

Recall the  function $\varrho$ defined in equation \eqref{varrho}. Corollary \ref{cor:varrhosmooth} states that $\varrho$ is smooth. Hence, we only need to prove that
\begin{align*}
\frac{\partial \varrho}{\partial \beta} (r, \beta_0) >0  \qquad \forall \beta_0 \in \mathbb{R}_{>0}.
\end{align*}
From Corollary \ref{cor:varrhosmooth} it follows that
\begin{align}
\lim_{r \rightarrow 0}\frac{\partial \varrho}{\partial \beta }(r, \beta_0)=0 \quad \forall \beta_0 \in \mathbb{R}_{>0}. \label{lem:proof:limvarrho}
\end{align}
Observe that:
\begin{align*}
\frac{\partial^2 \varrho}{\partial r \partial \beta} &= \frac{\partial  }{\partial \beta} \left(   \frac{\partial \varrho }{\partial r } \right) \\
&=  \frac{\partial  }{\partial \beta} \left( \frac{3}{2} \frac{1-\varrho^2}{r \varrho} \left( -  1+  \frac{ \sqrt{1-\varrho^2}  (   \beta  {r}^2 g(  {r}^2) + \varrho    )   }{{r}}    \right)   \right). 
\end{align*}
Then, it follows that
\begin{align}
\frac{\partial^2 \varrho }{\partial \beta \partial r}  \Big|_{\frac{\partial \varrho}{\partial \beta } =0 }>0. \label{lem:proof:monn}
\end{align}
From \eqref{lem:proof:limvarrho} and \eqref{lem:proof:monn} we get  $\frac{\partial \rho}{ \partial \beta} >0$.


\subsection{Proof of Lemma \ref{lem:Rb}}
\label{p:lem:Rb}

The proof of  Lemma \ref{lem:Rb} follows from an equilibrium study of toy model \eqref{toygov}. More specifically, the base limit given in S4 of Definition \ref{def:toytip} corresponds to a  unique equilibrium for every $\beta$. 

Let $\beta>0$. Recall that $g$ satisfies \eqref{gfuncc}. We define
\begin{align}
f(r, \beta) := \beta r g(r^2) -1 \label{fbeta}.
\end{align}
The function $f(\cdot, \beta)$ has a unique root since $f(0,\beta) <0$, $\lim_{r \rightarrow \infty}f(r, \beta) >0$ and $\frac{\partial f}{ \partial r} >0$.  Denote the root of $f_\beta$ by  $\tilde{R}(\beta)$. Observe that $(0,\tilde{R}(\beta))$ is an equilibrium of the toy model \eqref{toygov}. We will show that $ \frac{\partial \tilde{R}(\beta)}{\partial \beta}<0$. We have that
\begin{align*}
\frac{ \partial f(\tilde{R}(\beta), \beta)}{\partial \beta} &=  0 . \\
\frac{1}{\beta} + \beta \frac{\partial \tilde{R}(\beta)}{\partial \beta} \left( g(\tilde{R}(\beta)^2) +  \tilde{R}(\beta)^2 g'(\tilde{R}(\beta)^2) \right) &= 0.
\end{align*}
Using $g'>0$ from \eqref{gfuncc} we obtain that $  \frac{\partial \tilde{R}(\beta)}{\partial \beta}<0$. It follows from S4 in Definition \ref{def:toy} that $R(\beta) = \tilde{R}(\beta)$.  This completes the proof.


\section{Curvature steady tip growth solutions at tip}
\label{sec:BATSC}

The following lemma shows that a steady tip growth solutions as given in Definition \ref{defi:steadytipgrowth} satisfies the S1 condition of toy steady tip growth solutions as given in Definition \ref{def:toy}.

\begin{lem} If $x=(\rho,r,h,\Psi,z)$ is a steady tip growth solution as given in Definition \ref{defi:steadytipgrowth} then 
\begin{align}
\lim_{s \rightarrow 0} \frac{\sqrt{1-\rho(s)^2}}{r(s)} = \eta_0 >0. \label{lem:eq:rhocurve}
\end{align}
\end{lem}

 \begin{proof} Let $x=(\rho,r,h,\Psi,z)$ is a steady tip growth solution as given in Definition \ref{defi:steadytipgrowth}. Then, it follows from T1 and T2 of Definition \ref{defi:steadytipgrowth} that $\lim_{s \rightarrow 0} h'(s) = 0$. Then, substituting the $h$-equation of (\ref{finalfull}) in the previous limit and evaluating the terms using T1 we get \eqref{lem:eq:rhocurve}.  
 \end{proof}

\section{Toy proof extensions to BATS model}
\label{ap}

We will prove part of Conjecture \ref{con:X}, \ref{con:AB}. We first need a technical lemma:

\begin{lem}  Assume Conjecture \ref{con:xa} is true. Then, for all $\alpha \in Y_\mu$ we have that $x_\alpha(\, \cdot \, ; \mu) = (\rho_\alpha,r_\alpha,h_\alpha, \Psi_\alpha,z_\alpha)(\, \cdot \, ; \mu)$ satisfies
\begin{align}
\Psi_\alpha(s ; \mu) = \frac{\int^{s}_0 r_\alpha(\sigma) h_\alpha(\sigma)d\sigma}{\Gamma(r_\alpha(s),z_\alpha(s))}.  \label{psia}
\end{align}
\label{lem:Psi}
\end{lem}

\begin{proof}
Follows from the uniqueness of $x_\alpha(\, \cdot \, ; \mu)$ given by Conjecture \ref{con:xa} and the substitution of \eqref{psia} into the $\Psi$-equation in the BATS model \eqref{finalfull}.
\end{proof}

The following is part of Conjecture \ref{con:X}:

\begin{lem}  Assume Conjecture \ref{con:xa} is true. Let $Y \subset Y_\mu$ and assume that $X(Y) \neq \emptyset$. Then, for all $\alpha \in X_\mu(Y)$ the maximal existence interval of $x_\alpha(\, \cdot \, ; \mu)$ is $\mathbb{R}_{>0}$.
\label{ap:lem:X}
\end{lem}

\begin{proof}
Let $\alpha \in X_\mu(Y)$. For notational convenience we will omit the $\mu$-dependency and write  $x_\alpha$. Let $(\rho_\alpha, h_\alpha,\Psi_\alpha, z_\alpha, r_\alpha):= x_\alpha$.  We will use contradiction to prove the lemma. Suppose that the maximal existence interval of ${x}_{\alpha}$ is given by $(0,s_0)$ with $s_0>0$.  We will prove the following: 

\qquad \textbf{Claim:} $\lim_{s \rightarrow s_0} \Vert {x}_{\alpha}(s) \Vert= \infty$ 

We use contradiction to prove the claim. Suppose that $\lim_{s \rightarrow s_0}  {x}_{\alpha}(s) \in \partial M$ with  $M$ the phase space given in \eqref{M}. It follows by the definition of tip solutions, Definition \ref{defi:tip}, that
\begin{align}
\lim_{s \rightarrow s_0} \rho_{\alpha}(s) &= c_0 \in [0,1) \label{rhobetastarr2} \\
 \lim_{s \rightarrow s_0} r_{\alpha}(s) &>0 .
\end{align}
Observe that from Lemma \ref{lem:Psi} it follows that $\lim_{s \rightarrow s_0} \Psi_\alpha(s) >0$. The $h$-equation in the BATS model \eqref{finalfull} is linear in $h$, therefore, if  $\lim_{s \rightarrow s_0}  {x}_{\alpha}(s) \in \partial M$ then $\lim_{s \rightarrow s_0}  {h}_{\alpha}(s)>0$. Hence, we have that $\lim_{s \rightarrow s_0}  {x}_{\alpha}(s) \notin \partial M$.

We will show that all the components of ${x}_{\alpha}$ are bounded: 
\begin{itemize}
\item[-] {$\rho_{\alpha}$ is bounded:} This follows from \eqref{rhobetastarr2}.
\item[-] {$r_{\alpha}$ is bounded:} Recall from the BATS model \eqref{finalfull} that $r'=\rho$. Hence, $\rho_{\alpha}$ is bounded implies that $r_{\alpha}$ is bounded.  
\item[-] {$z_{\alpha}$ is bounded:}  The BATS model \eqref{finalfull} gives that $z'=\sqrt{1-\rho^2}$. Hence, $\rho_{\alpha}$ is bounded implies that $z_{\alpha}$ is bounded.
\item[-] {$h_{\alpha}$ is bounded:} From the BATS model it follows that
\begin{align*}
0 <h_{\alpha}' &< \left( \frac{r \gamma(\rho_{\alpha},z_{\alpha},r_{\alpha})}{\Gamma(z_{\alpha},r_{\alpha})} \right) h_{\alpha} .
\end{align*}
Since $\rho_{\alpha}$, $z_{\alpha}$ , $r_{\alpha}$ are bounded  and $r_{\alpha}$ increasing it follows that $h_{\alpha}$  is bounded.
\item[-] {$\Psi_{\alpha}$ is bounded:} 
Since $\rho_{\alpha}$, $h_{\alpha}$, $z_{\alpha}$ , $r_{\alpha}$ are bounded  and $r_{\alpha}$ increasing it follows by Lemma \ref{lem:Psi} that $\Psi_{\alpha}$ is bounded. 
\end{itemize}
Consequently, $\Vert {x}_{\alpha} \Vert$ is bounded which completes the contradiction.
\end{proof}

The following is part of Conjecture \ref{con:AB}:

\begin{lem} Assume Conjecture \ref{con:xa} is true. Then, for all open $Y \subset Y_\mu$ the sets $A_\mu(Y),B_\mu(Y)$ are disjoint and the set $A_\mu(Y)$ is open.
\label{ap:lem:AB} 
\end{lem}

\begin{proof}
For notational convenience we will omit the $\mu$-dependency and write  $x_\alpha$. Let $(\rho_\alpha, h_\alpha,\Psi_\alpha, z_\alpha, r_\alpha):= x_\alpha$.  We will first prove the disjointness of $A_\mu(Y),B_\mu(Y)$.  We will use contradiction. Suppose that $A_{\mu}(Y)  \cap  B_{\mu}(Y) \neq \emptyset$. Take $\alpha \in A_{\rm toy}  \cap  B_{\rm toy} $. Then there exists a $s_0 >0$ such that $\rho_\alpha(s) \rho_\alpha'(s) <0$ for all $s \in (0,s_0)$ and   $\rho_\alpha'(s_0)=\rho_\alpha(s_0) =0$. It then follows by the BATS model \eqref{finalfull} that $\lim_{s \rightarrow s_0} r_{\alpha}(s_0) = \infty$ but $r_{\alpha}$ is bounded since $|r_\alpha'| \leq 1$. Hence, $A_\mu(Y),B_\mu(Y)$ are disjoint. 

We continue with proving the openness of $A_\mu(Y)$. Let $\alpha \in A_\mu(Y)$. Take $s_0$ such that  $ \rho_\alpha(s)\rho_\alpha'(s) <0$ for all $s \in (0,s_0)$ and  $\rho_\alpha(s_0) =0$. It follows from the BATS model \eqref{finalfull} that $\rho_\alpha'(s_0) <0$. Then, the openness of $A_\mu(Y)$ follows by the continuity of the map $F_{\mu}$ of Conjecture \ref{con:xa}.
\end{proof}

\end{appendices}

\newpage 

\bibliographystyle{plain}
\bibliography{ref3}

\end{document}